\documentclass[10pt,a4paper]{amsart}
\usepackage{geometry}
\usepackage{amsmath}
\usepackage{amssymb}
\usepackage{amsthm}
\usepackage{xcolor}
\usepackage{verbatim}
\usepackage{tabularx}
\usepackage{mathrsfs}
\usepackage{appendix}
\usepackage[pdftex]{hyperref}  
\theoremstyle{theorem}
\newtheorem{theorem}{Theorem}[section]
\newtheorem{lemma}[theorem]{Lemma}
\newtheorem{proposition}[theorem]{Proposition}
\newtheorem{corollary}[theorem]{Corollary}
\newtheorem{theoremletter}{Theorem}

\theoremstyle{remark}
\newtheorem{example}{Example}[section]
\newtheorem{remark}[theorem]{Remark}
\numberwithin{equation}{section}

\theoremstyle{definition}
\newtheorem{definition}{Definition}[section]

\DeclareMathOperator{\diver}{div}

\DeclareMathOperator{\Ric}{Ric}

\DeclareMathOperator{\Hess}{Hess}
\DeclareMathOperator{\trace}{tr}
\DeclareMathOperator{\vol}{Vol}

\newcommand{\R}{\mathbb{R}}

\newcommand{\eps}{\varepsilon}

\newcommand{\tcb}{\textcolor{blue}}

\newcommand{\di}{\mathrm{d}}

\newcommand{\disp}{\displaystyle}

\newcommand{\Mo}{\mathring{M}}

\newcommand{\seg}{{\rm seg}}

\renewcommand{\div}{\diver}

\renewcommand{\phi}{\varphi}
\renewcommand{\emptyset}{\varnothing}

\makeatletter
\newcommand*\owedge{\mathpalette\@owedge\relax}
\newcommand*\@owedge[1]{
	\mathbin{
		\ooalign{
			$#1\m@th\bigcirc$\cr
			\hidewidth$#1\m@th\wedge$\hidewidth\cr
		}
	}
}
\makeatother

\title[Splitting and rigidity of sub-static spaces]{Some splitting and rigidity results for sub-static spaces}

\author{Giulio Colombo}
\address{Dipartimento di Matematica ``F. Enriques", Universit\`a degli Studi di Milano, Via Saldini 50, I-20133 Milano, Italy}
\email{giulio.colombo@unimi.it}

\author{Allan Freitas}
\address{Departamento de Matem\'{a}tica, Universidade Federal da Para\'{\i}ba, 58.051-900 Jo\~{a}o Pessoa, Para\'{\i}ba, Brazil}
\email{allan@mat.ufpb.br/allan.freitas@academico.ufpb.br}

\author{Luciano Mari}
\address{Dipartimento di Matematica ``F. Enriques", Universit\`a degli Studi di Milano, Via Saldini 50, I-20133 Milano, Italy}
\email{luciano.mari@unimi.it}
\author{Marco Rigoli}
\address{Dipartimento di Matematica ``F. Enriques", Universit\`a degli Studi di Milano, Via Saldini 50, I-20133 Milano, Italy}
\email{marco.rigoli@unimi.it}
\begin{document}

\maketitle
%
%


\begin{abstract}
In this paper we study the rigidity problem for sub-static systems with possibly non-empty boundary. First, we get local and global splitting theorems by assuming the existence of suitable compact minimal hypersurfaces, complementing recent results in the literature. Next, we prove some boundary integral inequalities that extend works by Chru\'sciel and Boucher-Gibbons-Horowitz to non-vacuum spaces. Even in the vacuum static case, the inequalities improve on known ones. Lastly, we consider the system arising from static solutions to the Einstein field equations coupled with a $\sigma$-model. The Liouville theorem we obtain allows for positively curved target manifolds, generalizing a result by Reiris.
\end{abstract}

\section{Introduction}
The purpose of this paper is to study some rigidity problems for static solutions to the Einstein field equations
\begin{equation}\label{einsteq}
	\Ric_{\hat{g}}+\left(\Lambda-\frac{1}{2}S_{\hat{g}}\right)\hat{g}=T.    
	\end{equation}
Here, $(\hat{M}^{m+1},\hat{g})$ is a Lorentzian manifold of dimension $m+1\geq 4$ with Ricci and scalar curvature, respectively, $\Ric_{\hat{g}}$ and $S_{\hat{g}}$; $T$ is the stress-energy tensor, which accounts for the distribution of matter, energy, and momentum in the manifold, and $\Lambda$ is the cosmological constant.
%
Looking for static solutions, that is, solutions possessing a timelike Killing field whose orthogonal distribution is integrable, leads to study the following warped product manifolds:
\begin{equation}\label{static_model}
\hat{M} = \mathbb{R} \times M, \qquad \hat{g} = - u^2 \, \di t \otimes \di t + g,
\end{equation}
where $(M^m,g)$ is a Riemannian manifold and $0 < u \in C^\infty(M)$. 

Physical reasons demand that $T$ satisfies the null energy condition (NEC), that is, $T$ is non-negative on null vectors. Condition (NEC) is satisfied by well-known  sources including electrostatic ones, scalar fields, and certain perfect-fluid models, and its geometric consequences are highlighted, for instance, by Penrose's Singularity Theorem \cite{penrose} (see also \cite[Page 263]{hawking}). Letting $\pi : (\hat M,\hat g) \to (M,g)$ be the projection onto the second factor, for each null vector $Y$ a computation gives
\begin{equation}\label{eq_Q}
T(Y,Y)=\left[\Ric -\frac{\Hess  u}{u}+\frac{\Delta u}{u}g\right](\pi_*Y,\pi_*Y),
\end{equation}
where \(\Ric \), \(\Hess \), \(\Delta \) are the Ricci curvature, the Hessian and the Laplacian on \((M^m, g)\), respectively. The (NEC) condition is therefore equivalent to
\begin{equation}\label{sub_static}
u \Ric  - \Hess  u + (\Delta  u) g \doteq uQ\geq 0,
\end{equation}

\begin{definition}
	A \emph{sub-static triple} $(M^m,g,u)$ is the data of a smooth, complete Riemannian manifold $(M^m,g)$ of dimension $m \ge 3$, possibly with boundary $\partial M$, and a solution $u \in C^\infty(M)$ to \eqref{sub_static} with $u>0$ in the interior $\mathring{M}$. If $\partial M\neq\emptyset$, we assume $u=0$ on $\partial M$. 
\end{definition}

\begin{remark}
	A priori, we do not assume that $Q$ can be extended continuously to $\partial M$ when this latter is non-empty. However, some result will need the condition, which is however meaningful due to \eqref{eq_Q} if we assume that $T$ is defined up to $\partial \hat M$.	
\end{remark}

In the above setting, $\partial M$ models the event horizon of a static black-hole. If $Q \equiv 0$, system \eqref{sub_static} describes static solutions to \eqref{einsteq} with $T \equiv 0$, named vacuum static spaces. A famous conjecture by Boucher, Gibbons, and Horowitz \cite{BGH} states:
\begin{quote}
\textit{(Cosmic No-Hair Conjecture) The only compact vacuum static triple \((M^m, g, u)\) with positive scalar curvature and connected boundary is given by a standard round hemisphere \(\mathbb{S}^m_{+}\), with static potential \(u\) being the height function.}
\end{quote}
In particular, the Cosmic No-Hair Conjecture suggests that, under certain conditions, the evolution of the universe leads to the dominance of a simple, highly symmetric geometry, ruling out more complicated ``hairy" solutions. Several works have established the conjecture under additional hypotheses. Reilly \cite[Theorem 4]{reilly} showed that the conjecture is true if $M$ is Einstein. Boucher, Gibbons, and Horowitz \cite{BGH} obtained a similar result, demonstrating the uniqueness of spacetime (the anti-de Sitter space) in the case of negative scalar curvature. Boucher \cite{bouch} and Friedrich \cite{fried} proved the result by assuming a Penrose compactification of spacetime, coupled with certain conditions at the conformal infinity. The same positive answer is obtained when $(M^{m},g)$ is conformally flat, as proved independently by Kobayashi \cite{koba} and Lafontaine \cite{laf}. Despite these developments, Gibbons, Hartnoll, and Pope \cite{ghp} constructed counterexamples to the Cosmic No-Hair Conjecture in dimensions $4 \leq m \leq 8$, and Costa, Di\'ogenes, Pinheiro and Ribeiro \cite{costa} constructed a simply-connected counterexample in any dimension $m \geq 4$. To our knowledge, the conjecture in its full generality is still open in dimension $m=3$.

It is natural to wonder how the behavior of \( \partial M = u^{-1}(0)\) can influence the rigidity of the entire manifold. If we assume that $Q$ extends continuously to $\partial M$, equation \eqref{sub_static} implies that $\partial M$ is totally geodesic, and thus \(|\nabla  u|\) is constant on each connected component of \(\partial M\). If \((M^m, g, u)\) is a compact vacuum static triple, then necessarily its scalar curvature $S$ is a positive constant which we can assume to be $m(m-1)$ by scaling. In this case, denoting by $\{\Sigma_i\}$, $1 \le i \le l$ the connected components of $\partial M$, Chru\'sciel \cite{Crus} showed that
\begin{equation}\label{crus}
\sum_{i=1}^l \kappa_i \int_{\Sigma_i} \left(S_{\Sigma_i} - (m-2)(m-1) \right) \geq 0,
\end{equation}
where $S_{\Sigma_i}$ is the scalar curvature of $\Sigma_i$ and \(\kappa_i\) is the restriction of \(|\nabla  u|\) to \(\Sigma_i\), named the surface gravity of $\Sigma_i$. Furthermore, equality implies that \(M\) is a round hemisphere and, consequently, that \(l = 1\), i.e., \(\Sigma\) is connected. Chrusciel's inequality is a key step to prove an important rigidity result in dimension \(m = 3\):
\begin{theorem}[Boucher, Gibbons, and Horowitz \cite{BGH}; Shen \cite{sh}]\label{bgh}
Let \((M^3, g, u)\) be a compact oriented vacuum static triple with connected boundary and scalar curvature $6$. Then the area of \(\partial M\) satisfies the inequality
\[
|\partial M| \leq 4\pi.
\]
Moreover, equality holds if and only if \((M^3, g)\) is isometric to the standard hemisphere.
\end{theorem}

In the vacuum static case, this kind of inequalities were considered and extended in several directions. Among the results, we quote a gap phenomenon for the boundary area in \cite[Theorem C]{ambrozio}, as well as refined area bounds relating to the virtual mass in \cite[Theorems 1.4 and 1.5]{borghini}. For general sub-static triples, we prove a family of inequalities extending \eqref{crus}, depending on a parameter $b$. The inequalities are a consequence of some identities in the spirit of Robinson’s \cite{robinson} and Shen's \cite{sh}, which may have independent interest. We obtain the following extension of \eqref{crus}:

\begin{theoremletter}\label{thm_BGH}
Let \((M^m, g, u)\) be a compact sub-static triple such that $\partial M \neq \emptyset$. Assume that $Q$ extends in a $C^1$-way to $\partial M$ and that the scalar curvature $S$ of $M$ satisfies
	\[
	S  - \trace  Q \qquad \text{is constant on } \, M.
	\]
	Then, $S  - \trace  Q > 0$. Moreover, for any real parameter
	\[
		b \geq - \frac{1}{m-1},
	\]
	the following inequality holds:

	\begin{equation} \label{BGHtype_intro}
		\sum_{i=1}^{l} \kappa_i^b \int_{\Sigma_i} \left(S_{\Sigma_i} - \frac{m-2}{m}S  - \frac{2}{m} \trace  Q\right) \ge 0
	\end{equation}
	where \(\Sigma_1, \dots, \Sigma_l\) are the connected components of \(\partial M\), $S_{\Sigma_i}$ is the scalar curvature of $\Sigma_i$ and \(\kappa_i > 0\) is the surface gravity of $\Sigma_i$. Moreover, the inequality in \eqref{BGHtype_intro} is strict unless \(M\) is isometric to a round hemisphere and \(Q \equiv 0\) on \(M\).
\end{theoremletter}

Exploiting the dependence on $b$ in the above result and applying the Gauss-Bonnet Theorem, in dimension \(m=3\) we get the following refinement of the inequality by Boucher-Gibbons-Horowitz, where $\partial M$ is allowed to be disconnected:

\begin{corollary}\label{coro_bgh}
	Let $(M^3,g,u)$ be a compact sub-static triple such that $\partial M \neq \emptyset$. Assume that $Q$ extends in a $C^1$-way to $\partial M$ and that the scalar curvature $S$ of $M$ satisfies
	\[
	S  - \trace  Q \qquad \text{is constant on } \, M.
	\]
	Then, $S  - \trace  Q >0$. Moreover, if $0 < \kappa_1 < \kappa_2 < \ldots < \kappa_j$ are the distinct surface gravities, decompose $\partial M = \bigcup_{a=1}^j \hat{\Sigma}_a$ where $\hat \Sigma_a$, possibly disconnected, has surface gravity $\kappa_a$. Then:
	\begin{itemize}
		\item[(i)] $\chi(\hat \Sigma_j) > 0$ and 
		\begin{equation}\label{ine_enhanced_BGH}
			|\hat{\Sigma}_j|\leq \frac{12\pi \chi(\hat \Sigma_j)}{S - \trace Q} \ ; 
		\end{equation}
		\item[(ii)] if, for some $i \in \{2,\ldots, j\}$, it holds 
		\begin{equation}\label{ine_enhanced_BGH_a}
			|\hat{\Sigma}_a|= \frac{12\pi \chi(\hat \Sigma_a)}{S - \trace Q} \qquad \forall \, i \le a \le j,
		\end{equation}
		then $Q \equiv 0$ on $\bigcup_{a=i}^j \hat \Sigma_a$, $\chi(\hat \Sigma_{i-1})>0$ and 
		\begin{equation}\label{ine_enhanced_BGH_am1}
			|\hat{\Sigma}_{i-1}|\leq \frac{12\pi \chi(\hat \Sigma_{i-1})}{S - \trace Q}  \ ;
		\end{equation}
		\item[(iii)] equality 
		\begin{equation}\label{ine_enhanced_BGH_tutti}
			|\hat{\Sigma}_a|= \frac{12\pi \chi(\hat \Sigma_a)}{S - \trace Q} \qquad \forall \, 1 \le a \le j
		\end{equation}
	holds if and only if $M$ is isometric to a round 3-hemisphere and $Q \equiv 0$ on $M$.
	\end{itemize}
\end{corollary}

\begin{remark}
	By observing that $\chi(\hat \Sigma_a) \le 2n_a$, where $n_a$ is the  number of components in $\hat \Sigma_a$ homeomorphic to $\mathbb{S}^2$, from the above we deduce the area bounds
	\[
	|\hat{\Sigma}_j|\leq \frac{24\pi n_j}{S - \trace Q}, \qquad |\hat{\Sigma}_{i-1}|\leq \frac{24\pi n_{i-1}}{S - \trace Q}, 
	\]		
	in (i) and (ii), respectively.
\end{remark}

We next investigate the validity of local and global splitting theorems for  sub-static triples. Our goal is to obtain effective results without requiring that $M$ is asymptotically flat or asymptotically hyperbolic, and without resorting on the specific form of the tensor $Q$. In this respect, our results differ from most of the extensive literature on no-hair theorems, see for instance \cite{agost, bcc, bunting, cerd, israel, robinson} and the references therein.

Given a sub-static triple $(M^{m},g,u)$, the analysis of the so-called optical metric $\bar{g}=u^{-2}g$ turns out to be a particularly useful tool. In this respect, Borghini and Fogagnolo \cite{fogborg} discovered a striking connection between the sub-static condition and the ${\rm CD}(0,1)$ condition: by setting $f = -(m-1) \ln u$, the manifold $(\mathring{M},\bar g)$ satisfies
\[
	\overline{\Ric}_f^1 \doteq \overline{\Ric} + \overline{\Hess} f + \frac{1}{m-1}\di f \otimes \di f = Q \ge 0
\]
that is, by definition, $(\mathring{M},\bar g, e^{-f}\di x_g)$ is a ${\rm CD}(0,1)$ space. A tightly related link was also observed in the work by Li and Xia \cite{lixia_17}, by means of an interpolating family of connections (see also the discussion in \cite[Appendix A.3]{fogborg}). We recall that, for $N \neq m$ and a function $f \in C^\infty(M)$, the Bakry-Emery Ricci tensor $\Ric_f$ and its $N$-modified counterpart $\Ric_f^N$ are defined as
\[
\Ric_f \doteq \Ric + \Hess f, \qquad \Ric_f^N \doteq \Ric + \Hess f - \frac{1}{N-m}\di f \otimes \di f.
\]
Recently, Wylie and Yeroshkin \cite{wylie1,wylie} developed a comparison theory for ${\rm CD}(0,1)$ spaces, and their results were adapted to sub-static manifolds to prove elegant splitting theorems in \cite[Theorems C, 3.7, 3.8 and Corollary 3.9]{fogborg} by assuming that the ends of the manifold are either {\it $u$-complete} or {\it conformally compact}. We recall the following

\begin{definition}
	Given a compact set $K\subset M$, an end $E$ with respect to $K$ is a connected component of $M\backslash K$ with non-compact closure. If $K$ is the image of a compact immersed submanifold $\Sigma \to M$, we will say that $E$ is an end with respect to $\Sigma$.
\end{definition}

\begin{definition}\label{def_u_complete}
	We say that an end $E\subset M$ (with respect to some compact set $K$) is \emph{$u$-complete} if, for any $g$-unit speed diverging curve $\gamma: [0, \infty) \longrightarrow \overline{E} \subset M$, it holds
	$$
		\int_0^\infty u^{-1}(\gamma(t))\di t = \infty \qquad \mbox{and}\qquad \int_{0}^{\infty}u(\gamma(t)) \di t= \infty.
	$$
\end{definition}
\noindent This is a mild request satisfied, for instance, if 
\[
C^{-1} r(x)^{-1} \le u(x) \le Cr(x) \qquad \text{for } \, x \in E, \ r(x) >>1, 
\]
where $C>1$ is a constant and $r$ is the distance from a fixed compact set (see Proposition \ref{prop_suffcond}). Definition \ref{def_u_complete} looks a bit different from that in \cite{fogborg}, see however Remark \ref{rem_ucomplete}. If the compact set $K$ is not too pathological (for instance, if $K$ is a Lipschitz embedded hypersurface), the property corresponds to the completeness of $\overline{E}$ endowed with the distance induced by any of the metrics $u^{-2}g$ and $u^2 g$.

The presence of area minimizing hypersurfaces often foresees splitting properties. By analyzing the flow introduced by Galloway in \cite[Lemma 3]{galloway}, Ambrozio \cite[Proposition 14]{ambrozio} (for $m=3$) and Huang, Martin, and Miao \cite{hmm} studied the local splitting behavior of vacuum static triples, see also Cruz, Lima and de Sousa  \cite{tiarlos} (for $m=3$) and Coutinho and Leandro \cite{benedito} for analogous results in electrostatic systems. One contribution of the present paper is to explicitly observe a connection between area-minimizing hypersurfaces in sub-static systems and $f$-area-minimizing hypersurfaces with respect to the conformal metric $\bar{g} = u^{-2}g$. This connection and the link between sub-staticity and the ${\rm CD}(0,1)$ condition allows us to extend the results in \cite{ambrozio,benedito,tiarlos,hmm} to any sub-static space, streamlining their proofs, and to complement those in \cite{fogborg}. We obtain:

\begin{theoremletter}\label{local_split}
	Let $(M^{m}, g, u)$ be a sub-static triple and $\Sigma \to M$ be a closed, 2-sided minimally embedded hypersurface such that $u>0$ on $\Sigma$. 
	\begin{itemize}
		\item[(A)] Assume that $\Sigma$ is connected and locally area minimizing. Then there is $\eps>0$ and a diffeomorphism 
		\[
			\Phi : (-\eps,\eps) \times \Sigma \to M
		\]
		such that in coordinates $(s,y) \in (-\eps,\eps) \times \Sigma$ it holds
		\[
			\Phi^* g = r^{2}(y)\di s^2 + h^{\Sigma},
		\]
		where $h^\Sigma$ is the induced metric on $\Sigma \to (M,g)$ and $r : \Sigma \to (0,\infty)$. Moreover, there exists a function $\sigma : (-\eps,\eps) \to (0,\infty)$ such that $u(\Phi(s,y)) = r(y)\sigma(s)$.
		\item[(B)] If an end $E$ with respect to $\Sigma$ is $u$-complete, then:
		\begin{itemize}
			\item[(i)] the topological boundary $\partial E \subset M$ is a single connected component of $\Sigma$, and separates $E$ from $M \backslash E$;
			\item[(ii)] the closure $\overline{E} \subset M$ is an embedded submanifold isometric to 
			\[
			[0,\infty) \times \partial E \qquad \text{with metric} \qquad g = r^{2}(y)\di s^2 + h^{\Sigma},
			\]
			where $h^\Sigma$ is the induced metric on $\Sigma \to (M,g)$ and $r : \Sigma \to (0,\infty)$. Moreover in coordinates $(s,y)\in [0,\infty) \times \partial E$, it holds $u(s,y) = r(y)\sigma(s)$ on $E$ for some function $\sigma : [0,\infty) \to (0,\infty)$.
		\end{itemize}	
	\end{itemize}
\end{theoremletter}

\begin{remark}\label{rem_covering}
	In (A), one can require that $\varsigma : \Sigma \to M$ is merely an immersion with image $\varsigma(\Sigma)$ area minimizing in the sense of Almgren (see \cite{almgren76,taylor76}), that is, satisfying $|\varsigma(\Sigma)| \leq |\psi(\varsigma(\Sigma))|$ for any Lipschitz continuous map $\psi : M \to M$ with $\psi(\varsigma(\Sigma)) \subset \mathring{M}$. In this case, in our assumptions $\varsigma$ turns out to factorize through a Riemannian covering $\Sigma \to \hat{\Sigma}$ and an embedding $\hat\Sigma \to M$, see Theorem \ref{teo_splitting_1}.
\end{remark}

Some remarks are in order:

\begin{itemize}
	\item In (A), the particular cases treated in \cite{ambrozio,benedito,tiarlos,hmm} allow for extra information on the function $r$. More precisely, in the vacuum static case it is shown in \cite{ambrozio} and \cite[Proposition 5]{hmm} that $r(y)$ is constant, thus $g$ is locally the product metric. The same conclusion is obtained in the presence of an electromagnetic field and for $m=3$, if the electric field vanishes identically, see \cite[Proposition 7]{tiarlos}. 
	\item The proof of (B) follows the standard splitting techniques by means of Busemann and distance functions, see in particular \cite{kasue}, adapted by Wylie \cite{wylie1} to ${\rm CD}(0,1)$ manifolds. However, despite \cite[Section 5]{wylie1} treats manifolds with boundary, our theorem cannot be obtained as an application of \cite{wylie1}. Likewise, the results in \cite{fogborg} do not apply in the setting of (B); note, in particular, that the $u$-completeness assumption in (B) is localized to a single end $E$. This is different from \cite{fogborg}, where \emph{every} end needs to be either $u$-complete or conformally compact. 
	\item The condition of $u$-completeness is highlighted, for instance, in a paper by Reiris  \cite{reiris_compa}. There, the author conjectured that a $3$-dimensional vacuum static triple $(M^3,g,u)$ with non-empty boundary and $u$-complete ends must either be the Schwarzschild solution or a flat solution, namely, $M = [0,\infty) \times \Sigma$ with the product metric $\di s^2 + h^\Sigma$ and potential $u(s,y) = s$. Here, $h^\Sigma$ is a flat metric on $\Sigma$. This would be a far-reaching extension of the known no-hair theorems for Schwarzschid space, in which an asymptotic flatness assumption is required (for example, see \cite{agost, bunting, cerd, israel, robinson}).
	\item If $\partial M = \emptyset$, in \cite{case} Case proved the triviality of complete, vacuum static triples with non-negative scalar curvature (meaning, $u$ is constant and $M$ is Ricci-flat) in any dimension. The result is somehow implicit in \cite{case}, but can be seen as follows: the substitution $u = e^{-f}$ transforms \eqref{sub_static} with $Q=0$ into
	\[
		\Ric_f^{m+1} = (\Delta_{f}f)g, \qquad \Delta_f f = \frac{S}{m-1}.
	\] 
	Taking into account that $S \ge 0$ is constant on a vacuum static space, by applying \cite[Theorem 1.5]{case} and letting $a\to \infty$ one gets the constancy of $u$. 
	\item Regarding the more general case of \emph{stationary} solutions to the Einstein vacuum equation, rigidity in dimension $m=3$ by only assuming $(M^3,g)$ complete was obtained by Anderson \cite{anderson}. A simpler proof under the further assumption that $\tilde{g}= u^2 g$ is complete was later given by Cortier and Minerbe \cite{cortier_minerbe}. The use of the metric $\tilde{g}$ is common in the literature. For vacuum static triples $(M^3,g,u)$, the system obtained by rewriting \eqref{sub_static} in the metric $\tilde{g}$ is called the harmonic map representation of $M$ (the reason for the name is apparent, for instance, in \cite{anderson}). In \cite{reiris_compa}, the author developed an interesting comparison theory for sub-static triples based on the use of $\tilde{g}$. It would be nice to clarify the interplay with the theory developed in \cite{fogborg}, as it may unveil further properties of sub-static spaces.
\end{itemize}

The last part of the paper focuses on a special sub-static triple: the one arising from the Einstein field equations on $\hat{M}$ with source a nonlinear $\sigma$-model, described by a map $\Phi: (\hat{M}^{m+1}, \hat{g}) \to (N^n, h)$ valued in a Riemannian manifold $(N,h)$ and a potential $V: (N^n, h) \to \mathbb{R}$. The associated gravitational+matter Lagrangian is
\[
	\mathcal{L}(\hat{g},\Phi) = \int_{\hat{M}} S_{\hat g} \di x_{\hat{g}} + \int_{\hat{M}} \left[ |\di\Phi|_{\hat{g}}^2 + (m-1)V(\Phi) \right] \di x_{\hat{g}}.
\]
For notational convenience, we incorporate the cosmological constant into $V$. Critical points of $\mathcal{L}$ solve \eqref{einsteq} with 
\[
	T = \Phi^{*}h - \frac{1}{2} \left( |\di \Phi|_{\hat{g}}^2 + (m-1)V(\Phi) \right)\hat{g}
\]
and the equation of motion $\delta_\Phi \mathcal{L} = 0$. Letting $\pi : (\hat M,\hat g) \to (M,g)$ be the projection onto the second factor, if $\Phi$ factorizes as $\varphi\circ\pi$ for some smooth map $\varphi : (M,g) \to (N,h)$, static solutions to $(\delta_{\hat g} \mathcal{L}, \delta_\Phi \mathcal{L}) = (0,0)$ therefore solve
\[
\left\{
		\begin{array}{r@{\;}c@{\;}l}
		u\Ric  - \Hess  u + (\Delta  u)g & = & u\varphi^{*}h, \\[0.2cm]
		-\Delta  u & = & V(\varphi) u, \\[0.2cm]
		u \tau(\varphi) + \di \varphi(\nabla u) & = & (m-1) \frac{D V(\varphi)}{2} u,
		\end{array}
\right.
\]
where $D$ is the Levi-Civita connection of $h$ and $\tau(\varphi)$ is the tension field of $\varphi$ (see Section 4 and \cite{ans21}). A relevant example is given by the Klein-Gordon field, for which
\[
	\varphi : M \to \mathbb{C}, \qquad V(\varphi) = c^2 |\varphi|^2, \quad c \in \R.
\] 
When the target manifold of $\varphi$ is $N = \R,\mathbb{C}$ and $M$ has no boundary, Reiris \cite{reiris} obtained sharp rigidity results by only assuming the completeness of the underlying manifold $M$ and suitable conditions on $V$, general enough to encompass the Klein-Gordon and other scalar fields of interest. His approach employs techniques from comparison geometry \`a la Bakry-\'Emery. In Theorem \ref{vanishing_map}, we allow for manifolds with $\partial M \neq \emptyset$ and targets $N$ be positively curved (in a controlled way). We obtain:

\begin{theoremletter}\label{vanishing_map}
	Let $(M^m, g, u)$ be a sub-static triple with either $\partial M = \emptyset$ or $\partial M$ compact, and let $\varphi : M \to (N,h)$ be a smooth map. Assume that 
	\begin{equation}\label{map_source}
		\left\{
		\begin{array}{r@{\;}c@{\;}l}
			u\Ric  - \Hess  u + (\Delta  u)g & \ge & u\varphi^{*}h, \\[0.2cm]
			-\Delta  u & = & V(\varphi) u, \\[0.2cm]
			u \tau(\varphi) + \di \varphi(\nabla u) & = & (m-1) \frac{D V(\varphi)}{2} u,
		\end{array}
		\right.
	\end{equation}
	and that 
	\[
	\inf_{M}V(\varphi)>-\infty, \qquad(m-1)\Hess V + 2Vh \geq -ah\qquad\mbox{on}\quad\varphi(M),
	\]
	for some constant $a\geq 0$. If the sectional curvatures of $N$ satisfy 
	\begin{equation}\label{est_sec}
		\sup_{N} \mathrm{Sec}_{N}\leq\kappa < \frac{1}{m-1},    
	\end{equation}
	for some constant $\kappa \ge 0$, then 
	$$
	\sup_{M}|\di\varphi|^{2}\leq \left[\frac{am}{2(1-(m-1)\kappa)}\right]^{\frac{1}{2}}.
	$$
	In particular, if $a=0$ then $\varphi$ is constant. 
\end{theoremletter}

For instance, by applying the result to wave maps with $V \equiv 0$, we get

\begin{theorem}\label{vanishing_map_2}
	Let $(M^m, g, u)$ be a sub-static triple with either $\partial M = \emptyset$ or $\partial M$ compact, solving 
	\begin{equation}\label{map_source_2}
		\left\{
		\begin{array}{r@{\;}c@{\;}l}
			u\Ric  - \Hess  u + (\Delta  u) g & \ge & u\varphi^{*}h  \\[0.2cm]
			-\Delta  u & = & 0, \\[0.2cm]
			u \tau(\varphi) + \di \varphi(\nabla u) & = & 0,
		\end{array}
		\right.
	\end{equation}
for some smooth map $\varphi : (M,g) \to (N,h)$ into a Riemannian manifold $N$. If $N$ satisfies 
	\begin{equation}\label{est_sec_2}
		\sup_{N} \mathrm{Sec}_{N} < \frac{1}{m-1},    
	\end{equation}
	then $\varphi$ is constant.\\
	In particular, a static solution \eqref{static_model} to Einstein's equation with $M$ complete, $\partial M$ compact or empty, and source a static wave map $\varphi : M \to N$ in a manifold whose curvature satisfies \eqref{est_sec_2} is vacuum static.
\end{theorem}

We emphasize that Theorems \ref{vanishing_map} and \ref{vanishing_map_2} do not impose any decay hypotheses such as asymptotic flatness. The general approach follows \cite{reiris} by performing the change of variables \( f = -\ln u \) and computing the Bochner formula for the \( f \)-Laplacian of \( |\di\varphi|^{2} \). However, Reiris obtained the desired Liouville theorem via gradient estimates which do not seem easy to extend to manifolds with boundary. For this reason, we proceed differently by means of integral maximum principles at infinity in the spirit of \cite{rigoli2}, which we shall apply to the energy density $|\di \varphi|^2$. In general, however, such results do not hold when $\partial M \neq \emptyset$ without growth assumptions on $|\di \varphi|^2$ on $\partial M$, a request we would rather avoid. The insight here is that the equality $u=0$ on $\partial M$ allows, perhaps surprisingly, to cancel boundary integrals. 

\vspace{0.5cm}

The paper is organized as follows. In Section 2, we discuss the equivalence between (NEC) and \eqref{sub_static}, and review the classical examples of sub-static triples. In Section 3, we prove Theorem \ref{local_split}. In Section 4, we establish some identities in the spirit of Robinson's \cite{robinson} and Shen's \cite{sh}, and prove Theorem \ref{thm_BGH}. The concluding Section 5 will be devoted to the proof of Theorem \ref{vanishing_map}.

\section{Preliminaries}

In this work, we consider a static spacetime $(\hat M, \hat g)$ of dimension $m+1$, where
\begin{equation}\label{static_space}
	\hat M = \R \times M \, , \qquad \hat g = - u^2 \di t \otimes \di t + g,    
\end{equation}
for $(M^m,g)$ a Riemannian manifold and $u \in C^\infty(M)$ a positive function. Let $\pi : \hat M \to M$ be the projection onto the second factor. Having fixed a local orthonormal frame $\{e_i\}$ on $M$ and the associated orthonormal frame $\{\hat e_0,\hat e_i\}$ on $\hat M$ with $\hat e_0 = \partial_t/u$ and $\pi_*\hat e_i = e_i$, the components of the Ricci tensor of $\hat M$ are given by
\begin{equation}\label{Ric_hat}
	\left\{\begin{array}{r@{\;}c@{\;}l}
		\hat R_{00} & = & \dfrac{\Delta  u}{u} \\[0.3cm]
		\hat R_{0i} & = & 0 \\[0.2cm]
		\hat R_{ij} & = & R_{ij} - \dfrac{u_{ij}}{u}.
	\end{array}
	\right.
\end{equation}
Here $R_{ij}$ and $u_{ij}$ are, respectively, the components of the Ricci tensor and of the Hessian of $u$ in $(M,g)$ in the frame $\{e_i\}$. Furthermore, the scalar curvatures satisfies
\begin{equation}\label{scal_hat}
	S_{\hat g} = S -\frac{2}{u}\Delta u  \, .    
\end{equation}
We particularly investigate a static spacetime that satisfies the Einstein field equation, expressed as
\begin{equation}\label{Einst_eq}
\Ric_{\hat{g}}+\left(\Lambda-\frac{1}{2}S_{\hat{g}}\right)\hat{g}=T,    
\end{equation}
where $T$ represents a stress-energy tensor, and $\Lambda$ denotes the cosmological constant. When expressed in coordinates, we can use \eqref{Ric_hat} and \eqref{scal_hat} to compute the components of the stress-energy tensor as follows:
\begin{equation}\label{energy}
	\left\{\begin{array}{r@{\;}c@{\;}l}
		T_{00} & = & -\Lambda+\dfrac{S}{2} \\[0.4cm]
		T_{0i} & = & 0 \\[0.2cm]
		T_{ij} & = & R_{ij} - \dfrac{u_{ij}}{u}+\left(\Lambda-\dfrac{S}{2}+\dfrac{\Delta u}{u}\right)\delta_{ij}
	\end{array}
	\right.
\end{equation}

\begin{definition}
A Lorentzian manifold $(\hat{M},\hat{g})$ satisfies the \textit{null energy condition} (NEC) if $T(Y,Y)\geq 0$ for all null vectors $Y$ (meaning $\hat{g}(Y,Y)=0$).
\end{definition}

To relate (NEC) to the sub-staticity of $(M,g,u)$, we follow the computations outlined, for instance, in \cite[Lemma 3.8]{wang}. Up to scaling, we can consider a null vector $Y = \hat{e}_0  + X$, where $X \in \mathscr{X}(M)$ satisfies $g(X,X)=1$. From \eqref{static_space} and $\hat{g}(\hat{e}_0 ,\hat{e}_0 )=-1$, applying the null energy  condition (NEC) to $Y$ yields:
\begin{eqnarray*}
0\leq T(Y,Y)&=&T_{00}+T_{ij}X^{i}X^{j}\\
&=&\left(-\Lambda+\frac{S }{2}\right)+\left[\Ric -\frac{\Hess  u}{u}+\frac{\Delta u}{u}g\right](X,X)+\left(\Lambda-\frac{S }{2}\right)g(X,X)\\
&=&\left[\Ric -\frac{\Hess  u}{u}+\frac{\Delta u}{u}g\right](X,X),
\end{eqnarray*}
which is \eqref{eq_Q}. Taking traces in \eqref{sub_static}, the scalar curvature $S$ of $M$ relates to $u$ as follows:  
\begin{equation}\label{sub_trace}
	\Delta  u = \frac{u}{m-1} (\trace  Q-S ),  
\end{equation}  

Sub-static triples include the following classical examples.

\begin{example}[\textbf{Vacuum static system}]
The simplest example arises from the Vacuum Static equation, where $T$ (and consequently $Q$) vanishes. In this case, a triple $(M^m, g, u)$ that satisfies  
\begin{equation}\label{eq_vacuum_static}
u \Ric - \Hess u + (\Delta u) g = 0
\end{equation}  
is called a \textit{vacuum static triple}. As stressed in \cite{fm75}, the left hand-side of vacuum static equation also appears as the formal adjoint of the linearized scalar curvature (up to sign), and therefore ties to the scalar curvature prescription problem. For more details, see also \cite[Section 4.2]{corvino}.
\end{example}

\begin{example}[\textbf{Electrostatic system}]
Given a function $\eta \in C^{\infty}(M)$ and a triple $(M^m, g, u)$, we consider
$$
Q := \frac{1}{u^2} \left(|\nabla \eta|^2 g - \di \eta \otimes \di\eta \right).
$$
Clearly, $Q \geq 0$ by the Cauchy-Schwarz inequality. This tensor gives rise to the so-called Electrovacuum static system, and $\eta$ relates to the electric potential (see, for example, \cite[Definition 1]{andrade}). For details of this construction from a specific stress-energy tensor, see \cite{tiarlos}.
\end{example}

\begin{example}[\textbf{Nonlinear $\sigma$-model}]
Let $\hat{M}^{m+1}$ be spacetime and $(N^n, h)$ a Riemannian manifold. If $\Phi: (\hat{M}^{m+1}, \hat{g}) \to (N^n, h)$ and $V: (N^n, h) \to \mathbb{R}$ are smooth maps, we consider the following stress-energy tensor:
$$T = \Phi^* h - \frac{1}{2} \left( |\di \Phi|_{\hat{g}}^2 + (m-1) V(\Phi) \right) \hat{g},$$
where $\Phi^* h$ is the pull-back of $h$.

In this case, the matter is described by a ``wave map'' $\Phi$ and a scalar potential $V$ (for details, see Section \ref{sec_map} and also \cite{CB,reiris,ans21}). Einstein's equation is then equivalent to
$$
\Ric_{\hat{g}} = \Phi^* h + V(\Phi) \hat{g}.
$$
In the static case, we assume that $\Phi$ factorizes as $\Phi = \phi \circ \pi$, with $\pi: \hat{M} \to M$ the natural projection and $\varphi: M \to (N^n, h)$. Then, the system rewrites as

\begin{equation*}
\left\{\begin{array}{r@{\;}c@{\;}l}
u\Ric  - \Hess u + (\Delta u)g & = & u\varphi^* h \\[0.2cm]
-\Delta  u & = & V(\phi) u, \\[0.2cm]
\end{array}
\right.
\end{equation*}
thus $Q = \varphi^* h$ is a positive semidefinite tensor, and $(M^m, g, u)$ forms a sub-static system.
\end{example}

\begin{example}[\textbf{Perfect fluid}]
In the theory of perfect fluids for a spacetime $(\hat{M}^{m+1}, \hat{g})$ (see, for example, \cite[Chapter III, Section 8]{CB}), the stress-energy tensor is given by
$$T = (\rho + P) v \otimes v + P \hat{g},$$
where $\rho$ is the energy density, $P$ is the pressure, and $v$ is a unit time-like covector field that represents the fluid's velocity. 

In particular, if $(M^{m+1}, \hat{g})$ is static as in \eqref{static_space}, we can write $T = \rho u^2 \di t^2 + P g$. Einstein's equation then reduces to
\begin{equation*}
\left\{\begin{array}{r@{\;}c@{\;}l}
R_{ij} - \frac{u_{ij}}{u} + \left(\Lambda - \frac{S }{2} + \frac{\Delta  u}{u}\right) \delta_{ij} & = & P \delta_{ij} \\[0.2cm]
-\Lambda + \frac{S }{2} & = & \rho,
\end{array}
\right.
\end{equation*}
and hence,
$$Q = (\rho + P) g.$$
Therefore, $(M, g,u)$ is a sub-static triple if $\rho + P \geq 0$.
\end{example}

\section{Compact minimal hypersurfaces and splitting in sub-static manifolds}

In this section, we describe the connection between stable minimal hypersurfaces in sub-static systems and $f$-minimal stable hypersurfaces in manifolds with $\overline{\Ric}_f^1 \geq 0$. This approach allows us to provide simpler proofs and extend the analysis presented in \cite{galloway} and \cite{hmm}, where static systems were thoroughly studied (see also \cite[Section 4.1]{tiarlos} and \cite{benedito} for electrostatic systems). 

Our focus is on the rigidity problem for sub-static triples $(M^m,g,u)$ that admit a closed (i.e. compact, boundaryless) minimal hypersurface contained within the interior $\mathring{M} = \{u>0\}$.

\begin{remark}
	If the set $\partial M= u^{-1}(0)$ is a minimal hypersurface (for instance, if $Q$ extends continuously to $\partial M$), as a consequence of the strong tangency principle for minimal hypersurfaces any connected minimal hypersurface either is a component of $\partial M$ or lies in $\Mo$.    
\end{remark}

Let $\bar g = u^{-2}g$, which is often called the optical metric in the physical literature. Choose $f = - (m-1)\ln u$. Then, the formulas relating the Ricci and Hessian tensors in the metrics $g$ and $\bar g$ yield 
\begin{equation}\label{eq_link_Q_Ricpsi}
	\disp \overline{\Ric}_{f}^1 =\Ric - \frac{\Hess u}{u}+\frac{\Delta u}{u}g\doteq Q.
\end{equation}
For details, see \cite{fogborg}. Whence, the sub-static condition is equivalent to 
\[
	\overline{\Ric}_f^1 \ge 0.
\]
Let $\Sigma \to (M,g)$ be a closed immersed hypersurface. Denote by $\di x$ and $\di \bar x$, respectively, the volume measures of $g$ and $\bar g$, and by $\di \sigma$, $\di \bar \sigma$ the induced volume densities on $\Sigma$. Note that $\di \bar x = u^{-m}\di x$, $\di \bar \sigma = u^{1-m}\di \sigma$. The natural weighted densities associated to $\overline{\Ric}_f^1$ are 
\[
	\di \bar x_f \doteq e^{-f}\di \bar x, \qquad \di \bar \sigma_f \doteq e^{-f} \di \bar \sigma \equiv \di \sigma.
\]
The latter identity implies that the weighted $(m-1)$-area of $\Sigma$ in $(M, \bar g)$ (which we name the $f$-area) is given by 
\begin{equation}\label{eq_magic}
	\overline{\vol}_f(\Sigma) \doteq \int_{\Sigma} \di \bar \sigma_f = \int_{\Sigma} \di \sigma = \vol(\Sigma).
\end{equation}
Denote by $\nu$ a local $g$-unit normal field along $\Sigma$, and $\bar \nu = u \nu$ the corresponding $\bar g$-unit normal. Stationary points of $\overline{\vol}_f$ are called $f$-minimal hypersurfaces, and are characterized by the vanishing of the $f$-mean curvature
\[
	\bar H_f \doteq \bar H + \di f(\bar \nu),
\]
where $\bar H$ is the mean curvature of $\Sigma \to (M, \bar g)$ in direction $\bar \nu$. The identity \eqref{eq_magic} establishes that 
\[
	\begin{array}{c}
	\text{$\Sigma \to (M,g)$ is minimal}\\
	\text{(resp. stable minimal,}\\
	\text{or area minimizing)}
	\end{array} \qquad \Longleftrightarrow \qquad 
	\begin{array}{c}
	\text{$\Sigma \to (M,\bar g,e^{-f}\di \bar x)$ is $f$-minimal}\\
	\text{(resp. stable $f$-minimal,}\\
	\text{or $f$-area minimizing)}
	\end{array}
\]
The second fundamental forms $A$ of $\Sigma \to (M,g)$ and $\bar{A}$ of $\Sigma \to (M,\bar{g})$ (in the directions $\nu$ and $\bar{\nu}$, respectively) satisfy the following relation:
\begin{equation}\label{second_conf}
\bar{A}(X,Y) = u^{-1} \left[ A(X,Y) +\di \ln u(\nu) g(X,Y) \right] \qquad \forall \, X,Y \in TM.
\end{equation}
Taking traces, the corresponding mean curvatures relate as follows:
\[
	\bar H = u \left[ H + (m-1) \di \ln u(\nu) \right].
\]
Whence, the $f$-mean curvature of $\Sigma \to (M, \bar g)$ is
\begin{equation}\label{eq_barHf_H}
	\bar H_f = u \left[ H + (m-1) \di \ln u(\nu) \right] - (m-1)u \di \ln u(\nu) = uH.
\end{equation}
in accordance to \eqref{eq_magic}. Assume now that $\Sigma \to (M,g)$ is a 2-sided minimal immersion. The $f$-stability operator $\bar L_f$ of $\Sigma$, viewed as an $f$-minimal hypersurface in $(M, \bar g)$, is given by
$$
	\bar L_f \doteq \bar{\Delta}_f + \|\bar{A}\|^{2}+\overline{\Ric}_f(\bar{\nu},\bar{\nu}),
$$
where $\|\cdot\|$ is the $\bar g$-norm and 
\[
	\bar \Delta_f \psi \doteq \bar \Delta \psi - \bar g(\bar \nabla f, \bar \nabla \psi) \qquad \forall \psi \in C^2(M)
\] 
(see for example \cite{detang}). From \eqref{second_conf} and the minimality of $\Sigma$ in $(M,g)$, we compute 
\begin{eqnarray*}
	\|\bar{A}\|^{2}&=& u^{2}[|A|^{2}+(m-1)(\di \ln u(\nu))^{2}]\\ &=&u^{2}|A|^{2}+\frac{1}{m-1}(\di f(\bar{\nu}))^{2},
\end{eqnarray*}
hence, using \eqref{eq_link_Q_Ricpsi},
\begin{equation}\label{l_psi_stability}
	\bar L_f = \bar{\Delta}_f +u^{2}|A|^{2} + \overline{\Ric}_f^1(\bar{\nu},\bar{\nu}) = \bar{\Delta}_f +u^{2}\big[|A|^{2} + Q(\nu,\nu)\big].
\end{equation}
We remember that a 2-sided $f$-minimal hypersurface $\Sigma$ is said to be $f$-stable if
\begin{equation}\label{stability_inequality}
	0 \le -\int_{\Sigma}\psi (\bar L_f \psi) \di \bar x_f = \int_\Sigma \Big\{ \|\di \psi\|^2 - u^{2}\big[|A|^{2} + Q(\nu,\nu)\big]\psi^2\Big\} \di \bar x_f, \qquad \forall \, \psi \in C^\infty_c(\Sigma).
\end{equation}
This parallelism allows us to provide a short proof of the following result that was shown by Huang-Martin-Miao \cite{hmm} in the static case.

\begin{proposition}\label{tot_geod}
	Let $\Sigma \to (M,g)$ be a (possibly disconnected) closed, 2-sided, stable minimal  hypersurface in a sub-static manifold $(M^{m},g,u)$. Then $\Sigma$ is totally geodesic and $Q(\nu,\nu) = 0$ along $\Sigma$.
\end{proposition}

\begin{proof}
	Substituting $\psi = 1$ into the stability inequality \eqref{stability_inequality} and using \eqref{l_psi_stability}, we obtain
	\[
		0 \leq -\int_{\Sigma} u^{2} \big[ |A|^{2} + Q(\nu,\nu) \big] \leq 0,
	\]
	which implies $|A|^{2} + Q(\nu,\nu) \equiv 0$.
\end{proof}

Denote by $\varsigma : \Sigma \to (M,g)$ the isometric immersion. We next prove (A) in Theorem \ref{local_split}, a local splitting theorem first obtained in \cite{hmm} in the static case. 

\begin{definition}\label{def_localmin}
	Hereafter, we say that $\varsigma$ is locally area-minimizing if there exists a neighbourhood $V \supset \varsigma(\Sigma)$ such that for every $K \subset \Sigma$ compact and for every variation  $\{\varsigma_t\}$ of $\varsigma$ with image in $V$ and variation vector field supported in $K$, the hypersurfaces $K_t = (K, \varsigma_t^*g)$ satisfy $\vol(K_t) \ge \vol(K)$.
\end{definition}

In the assumptions in (A) of Theorem \ref{local_split}, $\varsigma : \Sigma \to M$ can also be viewed as a locally $f$-area minimizing hypersurface in $(M, \bar{g}, e^{-f} d\bar{x})$. Let $\bar{\nu}$ be a choice of $\bar g$-unit normal field along $\Sigma$, and define
\begin{equation}\label{def_Phi_t}
	\Phi : \Sigma \times (-\eps,\eps) \to M, \qquad \Phi(x, t) = \overline{\exp}_{\varsigma(x)}\big(t \bar{\nu}(x)\big), \qquad \Sigma_t = \Phi(t,\Sigma),
\end{equation}
where $\eps$ is small enough so that $\Phi_t$ is an immersion for each $t$ and $\Sigma_0$ is $f$-area minimizing in $\Phi(\Sigma \times( -\eps, \eps))$. Along this flow, we use \eqref{l_psi_stability} to obtain the following at the point $(x,t)$:
\[
	\frac{\di}{\di t}\bar{H}_{f} = \bar{L}_f 1 = u^{2} \left( |A|^{2} + Q(\nu, \nu) \right) \geq 0.
\]
Therefore, since $\bar{H}_f = 0$ at $t = 0$, we have $\bar{H}_f(\cdot, t) \geq 0$ for all $t \in [0, \varepsilon)$. In particular, we have:
\[
	\overline{\vol}_f(\Sigma_{t}) - \overline{\vol}_f(\Sigma) = \int_0^t \left( - \int_{\Sigma_s} \bar{H}_f(\cdot, s) \, \di \bar{\sigma}_f \right) \di s \leq 0.
\]
This inequality cannot be strict for any $t \in (0, \varepsilon)$ because $\Sigma_0$ is $f$-area minimizing. Hence, $\bar{H}_f(\cdot, t) \equiv 0$ and $\overline{\vol}_f(\Sigma_t) = \overline{\vol}_f(\Sigma)$ for all $t \in [0, \varepsilon)$. The case $t \in (-\eps,0]$ is analogous. Thus, Proposition \ref{tot_geod} implies the following result in terms of the metric $g$:

\begin{proposition}\label{lam_split}
	Consider $(M^m, g, u)$ a sub-static system, and let $\Sigma$ be a (possibly disconnected) locally area-minimizing, closed, 2-sided minimal hypersurface where $u > 0$. Then, the family $\{\Sigma_t\}$ described above is totally geodesic, the area of $\Sigma_t$ is constant, and $Q(\nu_t, \nu_t) = 0$.
\end{proposition}

We are ready to prove (A) in Theorem \ref{local_split} in its more general form (see Remark \ref{rem_covering}). For the definition of a minimizer in Almgren's sense, see \cite{almgren76,taylor76}.

\begin{theorem}\label{teo_splitting_1}
	Let $(M^{m}, g, u)$ be a sub-static triple and $\varsigma : \Sigma \to M$ be a closed, connected, 2-sided minimally immersed hypersurface such that $u>0$ on $\Sigma$. Assume that $\Sigma$ is locally area minimizing. Then, 
	\begin{itemize}
		\item[(i)] If $\Sigma$ is embedded, there exists $\eps>0$ and a diffeomorphism $\Phi : (-\eps,\eps) \times \Sigma \to M$ such that, in coordinates $(s,y) \in (-\eps,\eps) \times \Sigma$,
		\[
			\Phi^* g = r^{2}(y)\di s^2 + h^{\Sigma},
		\]
		where $h^\Sigma$ is the induced metric on $\Sigma \to (M,g)$ and $r : \Sigma \to (0,\infty)$. Moreover, there exists a function $\sigma : (-\eps,\eps) \to (0,\infty)$ such that $u(\Phi(s,y)) = r(y)\sigma(s)$.
		\item[(ii)] If $\varsigma(\Sigma)$ is a minimizer in Almgren's sense, then $\varsigma$ factorizes as $\hat\varsigma \circ \pi$ with $\pi : \Sigma \to \hat \Sigma$ a Riemannian covering and $\hat\varsigma : \hat \Sigma \to M$ an embedding. Moreover, $\hat\Sigma$ is locally area-minimizing and thus (i) holds for $\hat\Sigma$. 
	\end{itemize}
\end{theorem}

\begin{proof}
	(i). Define $\Phi$, $\Sigma_t$, $\eps$ as in \eqref{def_Phi_t}, and choose $\eps$ so that $\Phi$ is a diffeomorphism onto a tubular neighbourhood of $\Sigma_0$. Proposition \ref{lam_split} asserts that $\Sigma_{t}$ is totally geodesic. The identity \eqref{second_conf} shows that $\Sigma_{t}$ is $\bar{g}$-totally umbilical, with second fundamental form
	\begin{equation}\label{eq_At}
		\bar{A}(X,Y) = \frac{1}{u} \di \ln u(\nu) g(X,Y) = \frac{\di u}{u}(\bar{\nu}) \bar{g}(X,Y).
	\end{equation}
	By the definition of $\Phi$, $\Sigma_{t}$ are the level sets of the $\bar{g}$-distance. Choose local coordinates $\{ y^\alpha\}$ on $\Sigma$ and write $\bar{h}^\Sigma = \bar{h}^\Sigma_{\alpha \beta} \di y^\alpha \otimes \di y^\beta$ for the induced metric on $\Sigma \hookrightarrow (M,\bar g)$. Write also
	\[
		\Phi^*\bar g = \di t^2 + \bar{h}_{\alpha\beta} \di y^\alpha \otimes \di y^\beta.
	\]
	Along the normal flow $\Phi$, we have the known variation formula
	\[
		\frac{\partial}{\partial t} \bar h_{\alpha\beta}(t,y) = -2\bar{A}_{\alpha\beta} = -2\frac{\partial}{\partial t}(\ln u) \bar h_{\alpha\beta}(t,y),
	\]
	and the initial condition $\bar h_{\alpha\beta}(0,y) = \bar h_{\alpha\beta}^{\Sigma}(y)$. Solving this ODE yields 
	\[
		\bar h_{\alpha\beta}(t,y) = \left(\frac{u(t,y)}{u(0,y)}\right)^{-2} \bar h_{\alpha\beta}^{\Sigma}(y),
	\]
	therefore we obtain
	\begin{equation}\label{eq_nice_split}
		\Phi^*\bar{g} = \di t^{2} + \left(\frac{u(t,y)}{u(0,y)}\right)^{-2} \bar h^{\Sigma}_{\alpha\beta}(y) \di y^\alpha \otimes \di y^\beta.
	\end{equation}
	From $\overline{\mathrm{Ric}}_f^1 \geq 0$ and applying a result of Wylie (see \cite[Proposition 2.2]{wylie1}), since $\Sigma$ is connected we have that
	\[
		u(t,y) = r(y)\xi(t).
	\]
	(without loss of generality, $\xi(0)=1$). Therefore, coming back to the metric $g$, we have
	\[
		\Phi^*g = r^2(y)\xi^2(t)\di t^{2} + h^{\Sigma}_{\alpha\beta}(y) \di y^\alpha \otimes \di y^\beta,
	\]
	where $h_\Sigma$ is the induced metric on $\Sigma \hookrightarrow (M,g)$. Choosing 
	\[
		s(t) = \int_0^t \xi(\tau) \di \tau, \qquad \sigma(s) = \xi(t(s)),
	\]
	we conclude the desired expression for $g$.\\[0.2cm]
	(ii) Assume that $\varsigma$ is not an embedding, so by compactness $\varsigma$ is not injective. Pick distinct points $x_1,x_2 \in \Sigma$ with $\varsigma(x_1)=\varsigma(x_2) = p$, and let $\delta < \min\{{\rm inj}(\Sigma),{\rm inj}(p)\}$ such that $\varsigma : B_\delta(x_j) \to M$ is a diffeomorphism onto its image. Up to composing with the inverse of the exponential map of $M$ at $p$, since $\Sigma$ is totally geodesic the images $V_j = \varsigma(B_\delta(x_j))$ are hyperplanes in $B_\delta(p)$. If they are transverse, then $\varsigma(\Sigma)$ cannot be a minimizer in the sense of Almgren. In fact, the local picture around codimension $1$ singularities of such minimizers can only consist of three hypersurfaces meeting at $120$ degrees (see Theorem 8.1 and subsequent discussion in \cite{delellis}). Hence, $V_1 \equiv V_2$ and $\hat \Sigma \doteq \varsigma(\Sigma)$ can therefore be given the structure of a smooth embedded manifold, with local charts $\exp_{x_j}^{-1} \circ \varsigma^{-1} : V_j \to \R^{m-1}$, and we can write $\varsigma = \hat \varsigma \circ \pi$, where $\hat\varsigma : \hat\Sigma \to M$ is the inclusion. The induced metric via $\hat \varsigma$ makes $\pi$ a local isometry. Since $\Sigma$ is compact, Ambrose's theorem implies that $\pi$ is a Riemannian covering. To conclude that $\hat \Sigma$ is locally area minimizing, simply observe that each variation $\hat\varsigma_t$ of $\hat\Sigma$ induces a variation $\varsigma_t = \hat \varsigma_t \circ \pi$ of $\Sigma$, and that $\vol(\Sigma_t) = k \vol(\hat\Sigma_t)$, with $k$ the number of sheets of $\pi$.
\end{proof}

\begin{remark}
	If $M$ is orientable, the assumption that $\varsigma(\Sigma)$ is minimizing in the sense of Almgren can be replaced by the requirement that $\Sigma$ is a mass-minimizing current. Indeed, the singular set of mass minimizing currents has Hausdorff codimension $7$, cf \cite{delellis}.
\end{remark}

In general, the above local splitting theorem cannot be extended to a global one, as shown by the simple example of a right cylinder $[-T,T] \times \mathbb{S}^{n-1}$ with two spherical caps attached to its boundaries in such a way that the resulting manifold has non-negative Ricci curvature (a sub-static manifold with $u \equiv 1$). Global splitting theorems were recently obtained in \cite[Theorems C, 3.7, 3.8 and Corollary 3.9]{fogborg}, and (B) in Theorem \ref{local_split} provides a further result in this direction. We begin by commenting on Definition \ref{def_u_complete}.

\begin{remark}\label{rem_ucomplete}
	Assume that $E$ is isometric to the interior of a manifold with boundary $E^*$. Since, by standard theory, the completeness of a given metric on $E^*$ is equivalent to the request that every diverging curve (i.e. eventually leaving any compact set) has infinite length, Definition \ref{def_u_complete} can be rephrased as the completeness of $(E^*, u^{-2}g)$ and of $(E^*, u^2g)$. By the Hopf-Rinow's theorem, the first condition in Definition \ref{def_u_complete} is also equivalent to say that
	\begin{equation}\label{ucomp_fogborg}
		\lim_{t\to \infty}\rho(\gamma(t))= \infty,
	\end{equation}
	where $\rho$ is the distance from a point in the metric $u^{-2}g$. In the setting of the present remark, our definition of $u$-completeness therefore agrees with the one in \cite[Definition 3.4]{fogborg}.
\end{remark}	

The following is a sufficient condition for the $u$-completeness of $E$, which improves on \cite[Proposition 3.2]{fogborg}

\begin{proposition}\label{prop_suffcond}
Let $(M^m,g,u)$ be a substatic triple and $E\subset M$ is an end. If
\[
	C^{-1}(1 + r)^{-1} \le u \le C( 1+ r) \qquad \text{on } \, E,
\]
for some constant $C>0$, where $r$ is the distance to a fixed compact set of $M$, then $E$ is $u$-complete.
\end{proposition}

\begin{proof}
	Let $K$ be the compact set in the statement. Suppose that $\gamma : [0,\infty) \to E$ be a unit speed divergent curve in $E$, say starting from a point of $\partial E$. By the triangle inequality, 
	\[
		t \ge \di (\gamma(t),\gamma(0)) \ge \di(\gamma(t), \partial E) \ge r(\gamma(t)) - \di(K, \partial E). 
	\]
	Therefore, since in our assumption both $u(\gamma(t))$ \and $u^{-1}(\gamma(t))$ are larger than $C^{-1}(1+r(\gamma(t)))^{-1}$, and since
	\[
		\int_0^\infty \frac{\di t}{1+r(\gamma(t))} \ge \int_0^\infty \frac{\di t}{1+ t + \di(K, \partial E)} = \infty,
	\]
	the thesis readily follows.
\end{proof}

We are ready to prove (B) in Theorem \ref{local_split}.

\begin{theorem}\label{teo_splitting_2}
	Let $(M^{m}, g, u)$ be a sub-static triple, and assume that  $\Sigma$ is a closed, embedded, 2-sided minimal hypersurface contained in $\Mo$. If an end $E$ with respect to $\Sigma$ is $u$-complete, then:
	\begin{itemize}
		\item[(i)] the topological boundary $\partial E \subset M$ is a single connected component of $\Sigma$, and it locally separates $E$ from $M \backslash E$.
		\item[(ii)] the closure $\overline{E} \subset M$ is an embedded submanifold isometric to 
		\[
			[0,\infty) \times \partial E \qquad \text{with metric} \qquad g = r^{2}(y)\di s^2 + h^{\Sigma},
		\]
		where $h^\Sigma$ is the induced metric on $\Sigma \to (M,g)$ and $r : \Sigma \to (0,\infty)$. Moreover in coordinates $(s,y)\in [0,\infty) \times \partial E$, it holds $u(s,y) = r(y)\sigma(s)$ on $E$ for some function $\sigma : [0,\infty) \to (0,\infty)$.
	\end{itemize}	 
\end{theorem}

\begin{proof}	
	The proof adapts the methods in \cite{kasue,wylie1,fogborg}. 
	In our assumptions, setting $\bar g = u^{-2}g$ we have that $(\bar{E},\bar g)$ is a complete metric space but in principle $(\mathring{M}, \bar g)$ may not be complete. Moreover, a component of the topological boundary $\partial E \subset M$ may not separate $E$ from $M \backslash E$. For this reason, we consider the metric completion $E^*$ of $(E,\bar g)$. Since $E$ is $u$-complete, $E^*$ is a manifold with boundary and $\partial E^*$ is isometric to the union of some of the connected components of $\Sigma$, possibly repeated twice. Indeed, components of $\Sigma$ that do not locally separate $E$ from $M \backslash E$ appear twice as components of $\partial E^*$.
	
	For convenience, we enlarge $E^*$ to a boundaryless manifold $V$ by gluing a collar neighbourhood $T \approx (0,\delta] \times \partial E^*$ of $\partial E^*$, and smoothly extend $\bar g$, $u$ to the entire $V$ in such a way that $u>0$ on $V$ and $(V,\bar g)$ is a complete manifold. Hereafter, we identify $E$ and $E^*$ as subsets of $V$. A $\bar g$-minimizing unit speed Lipschitz curve in $(V,\bar g)$ will be called a $\bar g$-segment.
	
	Let $\Gamma \subset E^*$ be either a point $p \in E$ or $\partial E^*$. We shall apply the weighted laplacian comparison theorems in \cite{wylie1,wylie} to the $\bar{g}$-distance to $\Gamma$ in suitable subsets of $(V,\bar g)$: 
	\[
		\rho : V \to \tcb{[}0,\infty), \quad \rho(x) \doteq \di_{\bar{g}}(x,\Gamma)
	\]
	
	However, a preliminary analysis of the regularity of $\rho$ is needed. \\[0.2cm]
	\noindent \textbf{Step 1}: regularity of $\rho$.\\[0.2cm]
	If $\Gamma = \partial E^*$, things are well-known: $\rho$ is smooth in a open subset $U \subset E^*$ (with the induced topology) containing $\partial E^*$, where the normal exponential map from $\partial E^*$ realizes a diffeomorphism; moreover, $U$ contains the interior of any segment from $\partial E^*$ to points in $E$, and $V \backslash U= {\rm cut}(\Gamma) \cap E$. 
		
	If $\Gamma = \{p\}$, Let $\overline{\exp}_p$ be the exponential map in $(V,\bar g)$ from $p \in E$. For each vector $v \in T_{p}V$, define $\gamma_v(t) = \overline{\exp}_p(tv)$ whenever the latter is defined, and following the notation in \cite[Section 5.7.3]{petersen} set
	\[
	\begin{array}{l}
	\seg_E(p) = \left\{ v \in T_pV : \begin{array}{l}
		\gamma_v([0,1)) \subset E \ \text{ and }\\
		\text{$\gamma_v$ is $\bar g$-minimizing on $[0,1]$}
		\end{array}\right\} \\[0.3cm]
	\seg_E^0(p) = \Big\{ tv  : v \in \seg_E(p), \ t \in [0,1)\Big\} 
	\end{array}	
	\]
	We observe, in particular, that for a vector $v\in\seg_E(p)$ the $\bar g$-geodesic $\gamma_v(t)$ does not touch $\partial E$ before time $t=1$. Whence, defining
		\[
		R(p) = \di_{\bar g}(p,\partial E^*), 
		\]
	and denoting by $B^{\bar g}_R$ metric balls in $(V, \bar g)$, it holds
	\[
		B^{\bar g}_{R(p)}(p) \subset \overline{\exp}_p\big(\seg_E(p)\big) , \qquad  \overline{\exp}_p\big(\seg_E^0(p)\big) \subset E.
	\] 
	Minor adaptations of \cite[Section 5.7.3]{petersen} enable to prove the following:
	\begin{itemize}
		\item[-] $\seg_E^0(p)$ is open and $\overline{\exp}_p$ is a diffeomorphism between $\seg_E^0(p)$ and its image $U \subset E$, the open set we search for. By construction, $\rho(x) = \|\overline{\exp}_p^{-1}(x)\|_{\bar g}$ is smooth on $U\backslash \{p\}$;
		\item[-] If $p,q \in E$, then $q\in\overline{\exp}_p(\seg_E^0(p))$ if and only if $p\in\overline{\exp}_q(\seg_E^0(q))$.
	\end{itemize}
	
	\vspace{0.2cm}
	\noindent \textbf{Step 2}: $f$-Laplacian comparison for $\rho$.\\[0.2cm]
	Suppose that $x \in U$, and let $\gamma : [0,\rho(x)] \to V$ be a $\bar g$-segment joining $\Gamma$ to $x$. By the construction of $U$, $\gamma((0,\rho(x)]) \subset U$ thus $\rho$ is smooth therein. Writing 
	\[
		\bar H(t) \doteq \bar \Delta\rho (\gamma(t)),
	\]
	it is known that $\bar H$ is the mean curvature of the level set $\{\rho = t\}$ at the point $\gamma(t)$ in direction $-\bar \nabla \rho = -\gamma'$, and that along $\gamma$ the following Riccati inequality holds (see \cite[Proposition 7.1.1]{petersen}):
	$$
		\bar{H}'+\frac{\bar{H}^2}{m-1}+ \overline{\Ric} (\gamma',\gamma') \le 0,
	$$
	where $'$ means derivative with respect to the $\bar g$-arclength $t$. The sub-static condition implies
	$$
		\bar{H}'+\frac{\bar{H}^{2}}{m-1}- \overline{\Hess} f(\gamma',\gamma')-\frac{1}{m-1}[(f \circ \gamma)']^{2} \leq 0.
	$$
	Observing that
	$$
		\bar{H}_f = \bar{H} -(f \circ \gamma)', \qquad \overline{\Hess} f(\gamma',\gamma') = (f \circ \gamma)''
	$$
	we have (leaving implicit the restriction to $\gamma$ for the ease of notation)
	\begin{eqnarray*}
		\bar{H}_f'=\bar{H}'-f''&\leq&-\frac{\bar{H}^{2}}{m-1}+f''+\frac{1}{m-1}(f')^{2}-f''\\
		&=&-\frac{1}{m-1}\left[\bar{H}_f+ f'\right]^{2}+\frac{1}{m-1}(f')^{2}\\
		&=&-\frac{\bar{H}_f^{2}}{m-1}-\frac{2}{m-1}f'\bar{H}_f.
	\end{eqnarray*}
	This implies that the function 
	\[
		\lambda(t)=e^{\frac{2f(\gamma(t))}{m-1}}\bar{H}_f(\gamma(t))
	\] 
	solves
	$$
		\frac{\di \lambda}{\di t} \leq- e^{-\frac{2f}{m-1}}\frac{\lambda^{2}}{m-1}.
	$$
	Changing variables according to 
	\[
		s(t) = \int_0^t e^{-\frac{2f(\gamma(\tau))}{m-1}}\di \tau = \int_0^t u^2(\gamma(\tau)) \di \tau
	\] 
	we rewrite the above inequality as
	\[
		\frac{\di \lambda}{\di s} \leq -\frac{\lambda^{2}(s)}{m-1}.
	\]
	One may now apply the Riccati comparison \cite{eschenburg} in two different cases:
	
	\begin{itemize}
		\item[-] if $\Gamma = \{p\}$, then from $\bar H(\gamma(t)) \sim \frac{m-1}{t}$ as $t \to 0$ we get
		\[
		\lambda(s) \sim \frac{m-1}{s} \qquad \text{as } \, s \to 0.
		\]
		Then, Riccati comparison gives 
		\[
		\lambda(s) \le \frac{m-1}{s}, \qquad \text{namely} \qquad \bar \Delta_f \rho (\gamma(t(s)) \le e^{\frac{-2f(\gamma(t(s)))}{m-1}} \frac{m-1}{s} 
		\]
		and, because of our choice of $f$, at the point $x$ it holds
		\begin{equation}\label{eq_point}
			\bar \Delta_f \rho (x) \le u(x)^2 \frac{m-1}{s(\rho(x))}
		\end{equation}
		\item[-] If $\Gamma$ is a hypersurface, denote by $o = \gamma(0)$. From 
		\[
		\lambda(0) = e^{\frac{2f(o)}{m-1}}\bar H_f(o), 
		\]
		Riccati comparison (see \cite[Theorem 3.8]{bmpr}) gives $s < s^* \doteq \frac{m-1}{-\lambda(0)}$ if $\lambda(0)<0$ and, regardless the sign of $\lambda(0)$,
		\[
		\lambda(s) \le \frac{\lambda(0)}{1+ \frac{\lambda(0)}{m-1}s} \le \lambda(0) \qquad \text{on } \, [0,s^*).
		\]
		The inequality rewrites as 
		\[
		\bar \Delta_f \rho (x) \le  e^{\frac{2[f(o)-f(x)]}{m-1}}\bar H_f(o) = \left( \frac{u(x)}{u(o)}\right)^2 \bar H_f(o),
		\]
		where we used our choice of $f$. In our assumptions, from \eqref{eq_barHf_H} and the minimality of $\Sigma$ in $(M,g)$ we deduce that $\Gamma = \partial E^*$ satisfies $\bar H_f = 0$. Therefore, 
		\begin{equation}\label{eq_Lapla_ipers}
			\bar \Delta_f \rho (x) \le \left( \frac{u(x)}{u(o)}\right)^2 \bar H_f(o) = 0.
		\end{equation}
	\end{itemize}
	
	\bigskip
	
	We shall now prove that \eqref{eq_point} and \eqref{eq_Lapla_ipers} hold in the barrier sense in the set
	\[
	E_\Gamma = \begin{cases}
		B^{\bar g}_{R(p)}(p)\backslash \{p\} & \quad \text{if } \, \Gamma = \{p\}, \\
		E & \quad \text{if } \, \Gamma = \partial E^\ast.
	\end{cases}
	\]
	It is enough to check the inequalities at points $x \in E_\Gamma \backslash U$. If $\Gamma=\{p\}$, by Calabi's trick, we can consider $\varepsilon>0$ and the function 
	$$
		\rho_{\varepsilon}\doteq\varepsilon+ \di_{\bar{g}}(\cdot,\gamma(\varepsilon)),
	$$
	which satisfies $\rho_{\varepsilon}\geq\rho$ in $M$ with equality at $x$. The smoothness of $\rho_\eps$ near $x$ follows since $\gamma(\eps) \in \overline{\exp}_x(\seg_E^0(x))$, which yields   $x\in\overline{\exp}_{\gamma(\eps)}(\seg_E^0(\gamma(\eps)))$ by Step 1. Applying Step 2 to $\di_{\bar{g}}(\cdot,\gamma(\varepsilon))$ we infer
	\[
		\bar \Delta_f \rho_\eps (x) \le u(x)^2 \frac{m-1}{s_\eps(\rho(x)-\eps)},
	\]	
	where 
	\[
	s_\eps(t) \doteq \int_0^t u^2(\gamma(\tau+\eps)) \di \tau = \int_\eps^{t+\eps} u^2(\gamma(\tau)) \di \tau
	\]
	and by the continuity of $s_\eps$ in $\eps$ we deduce our desired conclusion.
	
	Assume next that $\Gamma = \partial E^*$. Since $\partial E^*$ separates $V$, for $\varepsilon >0$ we can consider a compact hypersurface ${\Gamma}_{\varepsilon}\subset V\backslash E$ touching $\partial E^*$ at $o$ and whose second fundamental form $\bar{A}_{\varepsilon}$ in direction $-\gamma'(0)$ satisfies, at $o$,
	$$\bar{A}_{\varepsilon}=\bar{A}+\varepsilon\bar{g}.$$
	An explicit construction can be found in \cite[Lemma 2.1]{glmm}. Define
	$$
		\rho_{\varepsilon}= \di_{\bar{g}}(\cdot,\Gamma_{\varepsilon}).
	$$
	The separating property of $\partial E^*$ ensures that $\rho_{\varepsilon}\geq \rho$ in $E$, with equality at $x$. Also, $\rho_{\varepsilon}$ is smooth in a neighbourhood of $x$, as shown in \cite{glmm}. The computation leading to \eqref{eq_Lapla_ipers} applied to $\rho_{\varepsilon}$, guarantees that
	$$
		\bar \Delta_f \rho_{\varepsilon} (x) \le \left( \frac{u(x)}{u(o)}\right)^2 \left(\bar H_f(o)+ (m-1)\varepsilon\right),
	$$
	which shows that \eqref{eq_Lapla_ipers} holds in the barrier sense on $E$.  Observe  that the computation does not depend on the way $\bar g,u$ are extended on $V \backslash E^*$.\\[0.2cm]
		
	\noindent \textbf{Step 3}: Busemann functions and splitting.\\[0.2cm]
	Take a divergent sequence $x_j \in E$, $x_j \to \infty$ and $\bar g$-segments $\gamma_{j}:[0,t_{j}]\to V$ from $\partial E^*$ to $x_j$. By minimality, $\gamma_{j}((0,t_{j}])\subset E$.  Since $\partial E^*$ is compact the sequence subconverges to a $\bar g$-ray $\gamma : [0,\infty) \to E^*$ issuing orthogonally from $\partial E^*$, and satisfying $\gamma((0,\infty)) \subset E$. We consider the associated Busemann function $b_\gamma$ defined by
	$$
		b_\gamma(x) \doteq \lim_{t\to\infty}(\di_{\bar{g}}(x,\gamma(t))-t),
	$$
	for all $x\in V$. Observe that $\{\di_{\bar{g}}(\,\cdot\,,\gamma(t))-t\}$ is decreasing in $t$, whence
	\[
		\{ b_\gamma < 0\} = \bigcup_{t>0} B^{\bar g}_t(\gamma(t)) \qquad \text{and it is connected.}
	\]
	By the construction of $\gamma$, $B^{\bar g}_t(\gamma(t)) \subset E$ for all $t>0$, hence $\{b_\gamma < 0\} \subset E$. We shall prove that
	\begin{equation} \label{eq_lapla_buse}
		\bar\Delta_f b_\gamma \leq 0 \qquad \text{in the barrier sense on } \, \{b_\gamma < 0\}.
	\end{equation}
	Let us fix $x\in \{b_\gamma < 0\}$ and let $t_0 > 0$ be large enough so that $x \in B^{\bar g}_t(\gamma(t))$ for all $t>t_0$. For any fixed $t>t_0$, pick a $\bar g$-segment $\sigma_t : [0,\di_{\bar g}(x,\gamma(t))] \to \overline{B^{\bar g}_t(\gamma(t))} \subset E^*$ from $x$ to $\gamma(t)$. The family $\sigma_{t}$ subconverges as $t\to\infty$ to a $\bar{g}$-ray $\sigma:[0,\infty)\longrightarrow E^*$, which we claim to be contained in $E$. Indeed, as an application of the triangle inequality the Busemann function $b_\sigma$ associated to $\sigma$ satisfies
	\begin{equation} \label{bgammasigma}
		b_\gamma \leq b_\gamma(x) + b_\sigma
	\end{equation}
	with equality at $x$, see \cite[Proposition 7.3.8]{petersen}. Hence
	\[
		\{b_\sigma \leq 0\} \subset \{b_\gamma < 0\} \subset E
	\]
	The conclusion follows since $\sigma([0,\infty)) \subset \{b_\sigma \leq 0\}$ by the very definition of $b_\sigma$. By \eqref{bgammasigma}, the family of functions $b_{\sigma,t}(\cdot) \doteq b_\gamma(x) + \di_{\bar g}(\cdot, \sigma(t)) - t$ are barriers at $x$. Moreover, since $B^{\bar g}_t(\sigma(t)) \subset \{b_\sigma < 0\} \subset E$, by Step 1 the functions $\rho_t = \di_{\bar g}(\cdot,\sigma(t))$ and thus $b_{\sigma,t}$ are smooth near $x$. Setting
	\[
		s_{t}(x) \doteq \int_{0}^{\rho_t(x)} u^2(\sigma(t-\tau)) \di \tau = \int_0^t u^2(\sigma(\tau)) \di\tau,
	\]
	Step 2 gives
	\begin{equation}\label{lapl_est2}
		(\bar{\Delta}_f b_{\sigma,t})(x) = (\bar{\Delta}_f\rho_{t})(x)\leq u(x)^2 \frac{m-1}{s_{t}(x)}. 
	\end{equation}
	Note that $s_t(x) \to \infty$ as $t\to\infty$ by $u$-completeness of $E$, as $\sigma$ is a divergent curve in $E$. This proves \eqref{eq_lapla_buse}.
	
	Consider the distance $\rho$ to $\partial E^*$ and the function
	\[
		w\doteq \rho + b_\gamma \qquad \text{in } \, E. 
	\]
	We claim that 
	\begin{equation}\label{eq_smp_w}
		w \ge 0 \quad \text{in } \, E, \qquad w \equiv 0 \quad \text{along } \, \gamma. 
	\end{equation}
	First, $w(\gamma(t)) = t - t = 0$ for each $t \ge 0$. On the other hand, if $x \in E$, pick a $\bar g$-segment from $\partial E^*$ to $x$ and let $o \in \partial E^*$ be its initial point, and then a segment from $x$ to $\gamma(t)$. We compute
	\[
		\rho(x) + \rho_{\gamma(t)}(x) = \di_{\bar g}(x,o) + \di_{\bar g}(x,\gamma(t)) \ge \di_{\bar g}(o,\gamma(t)) \ge \di_{\bar g}(\partial E^*,\gamma(t)) = t,
	\]
	whence $\rho(x) + (\rho_{\gamma(t)}(x)-t) \ge 0$. The thesis follows by letting $t \to \infty$. By using \eqref{eq_Lapla_ipers} and \eqref{eq_lapla_buse} we get 
	\[
		\bar\Delta_f w \le 0 \qquad \text{in the barrier sense on } \, \{b_\gamma < 0\},
	\]
	thus from \eqref{eq_smp_w} and since $\gamma((0,\infty)) \subset \{b_\gamma<0\}$ we infer $w \equiv 0$ on $\{b_\gamma<0\}$ by the strong maximum principle in \cite{calabi} (see also \cite[Theorem 7.1.7]{petersen}). On the other hand, by the very definition of $w$ the set $\{w = 0\} \cap E$ is  closed in $E$ and contained in $\{b_\gamma<0\}$, whence by the above $\{b_\gamma<0\} \equiv \{w = 0\} \cap E$ is both open and closed in $E$. Being $E$ and $\{b_\gamma<0\}$ connected, it holds $E \equiv \{b_\gamma < 0\}$ and thus $w \equiv 0$ on the entire $E$. In particular,
	$$
		\bar{\Delta}_{f}\rho=\bar{\Delta}_{f}b_\gamma =0 \qquad \text{in the barrier sense on } \, E.
	$$ 
	and $\rho,b_\gamma$ are smooth\footnote{Indeed, solutions in the barrier sense of $\bar \Delta_f u = 0$ also solve the equation in the viscosity sense and weak sense (see \cite{mantemasceural}, \cite{ishii}), thus regularity follows by standard elliptic theory.}. 
	To see that $\bar{\nabla}\rho$ is the splitting direction, and the corresponding metric is warped as in the statement, we proceed as in \cite[Theorem C]{fogborg} (see also \cite[Lemma 3.5]{wylie1}). We consider the normal exponential map
	\[
		\Phi :[0,\infty) \times \partial E^* \longrightarrow (E^*,\bar{g}), \qquad \Phi(t,y) = \overline{\exp}_{y}(t\bar{\nu}(y)),
	\]
	and let 
	\[
	T = \sup \left\{ t>0 \ : \ (\Phi_t)_* \ \text{ is nonsingular on $[0,t) \times \partial E$}\right\}.
	\]
	As before, we write
	\[
	\Phi^*\bar{g} = \di t^2 + \bar{h}_{\alpha\beta} \di y^\alpha \otimes \di y^\beta \qquad \text{on } \, [0,T) \times \partial E.
	\]
	The weighted Bochner formula applied to $\bar{\nabla} \rho$ and the identity
	\[
	0 = \bar{\Delta}_f \rho = \bar \Delta \rho - \bar{g}(\bar{\nabla} f, \bar{\nabla} \rho)
	\]
	imply that
	\[
	0 = \frac{1}{2}\bar{\Delta}_{f} \|\bar{\nabla}\rho\|^{2} \geq \|\overline{\Hess}\rho \|^{2} - \frac{1}{m-1} \bar{g}(\bar{\nabla} f, \bar{\nabla} \rho)^{2} \ge 0,
	\]
	(the last inequality is Newton's one, once we recall that $\overline{\Hess} \rho(\bar \nabla \rho,\cdot) = 0$). Hence,
	\[
	\overline{\Hess} \rho = \frac{\bar{g}(\bar{\nabla} f, \bar{\nabla} \rho)}{m-1} \bar{h} = - \bar{g}(\bar{\nabla} (\ln u), \bar{\nabla} \rho) \bar{h} \qquad \text{on } \, \bar \nabla \rho^\perp \times \bar \nabla \rho^\perp.
	\]
	The level sets of $\rho$ are therefore totally umbilic, and the identities 
	\[
	\Phi^*\bar{g} = \di t^{2} + \left(\frac{u(t,y)}{u(0,y)}\right)^{-2} \bar h^{\Sigma}, \qquad u(t,y) = r(y)\xi(t) \qquad \text{on } \, [0,T) \times \partial E.
	\]
	follow as in the proof of Theorem \ref{teo_splitting_1}. If $T < \infty$, we thus observe that $\Phi_t^* \bar g$ is non-degenerate on $\{t = T\}$, contradicting the definition of $T$. Thus, $T = \infty$ and $\Phi$ is a local diffeomorphism on $[0,\infty) \times \partial E^*$.
	
	In order to prove that $\Phi$ is bijective, the surjectivity easily follows by the completeness of $(V,\bar g)$. Regarding injectivity, if $\Phi(t_1, y_1) = \Phi(t_2, y_2)$, we can apply $\rho$ to both sides and conclude that $t_1 = t_2 = t$. Moreover, since $\Phi_t: \left\{0\right\} \times \partial E^* \longrightarrow \left\{t\right\} \times \partial E$ is a diffeomorphism, being the flow map of the smooth vector field $\bar \nabla \rho$, we conclude from $\Phi(t, y_1) = \Phi(t, y_2)$ that $y_1 = y_2$. \\
	Having proved that $\Phi$ is an isometry, since $E^*$ is connected then $\partial E^*$ is connected. Hence, the topological boundary of $E$ is connected and separates $E$ from $M \backslash E$. Therefore, the closure $\bar{E}$ in $M$ is a manifold with boundary, isometric to $E^*$. This concludes the proofs of (i) and (ii). 	
\end{proof}

\section{A Boucher-Gibbons-Horowitz-type inequality for sub-static manifolds}

In this section we prove Theorem \ref{thm_BGH}. Let $(M^m,g,u)$ be a sub-static triple, and assume that $Q$ extends continuously to $\partial M$. This is equivalent to require that 
\[
\dfrac{\Hess u}{u} \text{ extends continously up to } \partial M.
\]  
\begin{remark}
In particular, since \( \partial M = u^{-1}(0) \), the boundary \( \partial M \) is a totally geodesic hypersurface, and \( |\nabla u| \) is constant along \( \partial M \). If \( \Sigma_1, \dots, \Sigma_k \) are the connected components of \( \partial M \), the positive constants \( \kappa_i = |\nabla u|_{\Sigma_i} > 0 \), for \( i = 1, \dots, k \), are called the surface gravities.
\end{remark}
By taking the trace of the equation above, we obtain
\[
uS = u \, \trace Q - (m-1) \Delta u.
\]

\begin{lemma}
	In the above assumptions, we have
	\begin{equation} \label{Qdu}
		2 Q(\nabla u,\,\cdot\,) = u (\di S - 2 \div Q) \qquad \text{in } \, \mathring{M} \, .
	\end{equation}
	In particular, if $S\in C^1(M)$ and $\div Q\in C^0(M)$ then
	\begin{equation} \label{Qdu0}
		Q(\nabla u,\,\cdot\,) = 0 \qquad \text{on } \, \partial M \, .
	\end{equation}
\end{lemma}

\begin{remark}
	Clearly, \eqref{Qdu0} still holds if one requires only that $(\di S - 2\div Q)$ remains bounded in a neighbourhood of $\partial M$. Furthermore, \eqref{Qdu} can be equivalently expressed as
	\begin{equation} \label{Qdu1}
		(m-1) \di\frac{\Delta u}{u} = -\frac{2}{u} Q(\nabla u,\,\cdot\,) - 2\div Q + \di\trace Q \qquad \text{in } \, \mathring{M}.
	\end{equation}
\end{remark}

\begin{proof}
	Recalling Schur's and Ricci's identities
	\begin{align*}
		\div \Ric & = \frac{1}{2} \di S \\
		\Ric(\nabla u,\,\cdot\,) & = \div\Hess u - \di \Delta u \\
		& \equiv \div(\Hess u - (\Delta u)g)
	\end{align*}
	we have
	\begin{align*}
		\div(uQ) & = \div(u\Ric) - \div(\Hess u - (\Delta u)g) \\
		& = u\div\Ric + \Ric(\nabla u,\,\cdot\,) - \div(\Hess u - (\Delta u)g) \\
		& = u\div\Ric \\
		& = \frac{u}{2} \di S
	\end{align*}
	in $\mathring{M}$. Since $\div(u Q) = u \div Q + Q(\nabla u,\,\cdot\,)$, we prove \eqref{Qdu}.
\end{proof}

\begin{lemma}
	In the above setting, we have
	\begin{equation} \label{Sh1}
		\frac{1}{2} \div\left(\frac{\nabla|\nabla u|^2}{u}\right) = \frac{|\Hess u|^2}{u} + \frac{Q(\nabla u,\nabla u)}{u} + \left\langle\nabla u,\nabla\frac{\Delta u}{u}\right\rangle
	\end{equation}
	and
	\begin{equation} \label{Sh2}
		\begin{split}
			\div\left(\frac{\mathring{\Hess u}(\nabla u,\,\cdot\,)^\sharp}{u}\right) & = \frac{|\mathring{\Hess u}|^2}{u} + \frac{Q(\nabla u,\nabla u)}{u} + \frac{m-1}{m} \left\langle\nabla u,\nabla\frac{\Delta u}{u}\right\rangle 
		\end{split}
	\end{equation}
	where $\mathring{\Hess u} = \Hess u - \frac{1}{m}(\Delta u) g$ is the traceless part of $\Hess u$.
\end{lemma}

\begin{proof}
	Recalling that $\frac{1}{2} \di|\nabla u|^2 = \Hess u(\nabla u,\,\cdot\,)$, an application of Bochner's formula
	\[
		\frac{1}{2} \Delta|\nabla u|^2 = |\Hess u|^2 + \Ric(\nabla u,\nabla u) + \langle\nabla u,\nabla\Delta u\rangle
	\]
	and of \eqref{sub_static} yields
	\begin{align*}
		\frac{1}{2} \div\left(\frac{\nabla|\nabla u|^2}{u}\right) & = \frac{1}{2} \frac{\Delta|\nabla u|^2}{u} - \frac{1}{2} \frac{\langle\nabla|\nabla u|^2,\nabla u\rangle}{u^2} \\
		& = \frac{|\Hess u|^2}{u} + \frac{\Ric(\nabla u,\nabla u)}{u} + \frac{\langle\nabla u,\nabla\Delta u \rangle}{u} - \frac{\Hess u(\nabla u,\nabla u)}{u^2} \\
		& = \frac{|\Hess u|^2}{u} + \frac{\langle\nabla u,\nabla\Delta u \rangle}{u} + \frac{Q(\nabla u,\nabla u)}{u} - \frac{|\nabla u|^2\Delta u}{u^2} \\
		& = \frac{|\Hess u|^2}{u} + \frac{Q(\nabla u,\nabla u)}{u} + \left\langle \nabla u,\nabla\frac{\Delta u}{u} \right\rangle ,
	\end{align*}
	that is, \eqref{Sh1}. Then, from the identities
	\begin{align*}
		\div \left( \frac{\Delta u}{u} \nabla u \right) & = \frac{(\Delta u)^2}{u} + \left\langle \nabla u,\nabla\frac{\Delta u}{u} \right\rangle \\
		|\Hess u|^2 & = |\mathring{\Hess u}|^2 + \frac{(\Delta u)^2}{m} \\
		\mathring{\Hess u}(\nabla u,\,\cdot\,)^\sharp & = \Hess u(\nabla u,\,\cdot\,)^\sharp - \frac{\Delta u}{m} \nabla u = \frac{1}{2} \nabla|\nabla u|^2 - \frac{\Delta u}{m} \nabla u
	\end{align*}
	together with \eqref{Sh1} we obtain
	\begin{align*}
		\div\left(\frac{\mathring{\Hess u}(\nabla u,\,\cdot\,)^\sharp}{u}\right) & = \frac{1}{2} \div\left(\frac{\nabla|\nabla u|^2}{u}\right) - \frac{1}{m} \div\left(\frac{\Delta u}{u} \nabla u\right) \\
		& = \frac{|\Hess u|^2}{u} - \frac{1}{m} \frac{(\Delta u)^2}{u} + \frac{Q(\nabla u,\nabla u)}{u} + \left(1-\frac{1}{m}\right) \left\langle \nabla u,\nabla\frac{\Delta u}{u} \right\rangle \\
		& = \frac{|\mathring{\Hess u}|^2}{u} + \frac{Q(\nabla u,\nabla u)}{u} + \frac{m-1}{m} \left\langle \nabla u,\nabla\frac{\Delta u}{u} \right\rangle
	\end{align*}
	proving \eqref{Sh2}. 
\end{proof}

On the set $M_0 := \{x \in M : \nabla u(x)\neq 0\}$, we define
\[
	\nu = -\frac{\nabla u}{|\nabla u|} \, .
\]
For each $p\in M_0$, the intersection $\Sigma = \Sigma_{u(p)} \cap \Omega_p$ of the level set $\Sigma_{u(p)} = \{x \in M : u(x) = u(p)\}$ with a sufficiently small neighbourhood $\Omega_p\subseteq M_0$ of $p$ is an embedded smooth hypersurface of $M$ with normal vector field $\nu$. We denote by $g_\Sigma$ the metric on $\Sigma$ induced by $g$,
\[
	g_\Sigma = g - \nu^\flat \otimes \nu^\flat \, .
\]
We also denote by $H$ the unnormalized mean curvature of $\Sigma$ in the direction of $\nu$,
\[
	H \doteq -\div\nu
\]
and by $A$ the second fundamental form of $\Sigma$ in the same direction. We have
\[
	A(X,Y) = \frac{\Hess u(X,Y)}{|\nabla u|} \qquad \forall \, X,Y \in T_x\Sigma \, , \, x \in \Sigma
\]
and also
\begin{equation} \label{H_Hessu}
	H = \trace_{g_\Sigma} A \equiv \frac{\Delta u - \Hess u(\nu,\nu)}{|\nabla u|} \, .
\end{equation}
We also set
\[
	\mathring{A} = A - \frac{H}{m-1} g_\Sigma
\]
for the traceless (with respect to $g_\Sigma$) part of $A$.

\begin{lemma}
	For any $C^2$ function $u$ on a Riemannian manifold $M$, at any point where $\nabla u\neq 0$ we have, with the notation introduced above,
	\begin{equation} \label{Hu0}
		|\mathring{\Hess u}|^2 = |\nabla u|^2 |\mathring{A}|^2 + \frac{m-2}{m-1}|\nabla^\top|\nabla u||^2 + \frac{m}{m-1} |\mathring{\Hess u}(\nu,\,\cdot\,)^\sharp|^2
	\end{equation}
\end{lemma}

\begin{proof}
	With respect to an orthonormal basis $\{e_i\}$ for $T_x M$ with $e_1 = \nu$, we have
	\begin{align*}
		|\mathring{\Hess u}|^2 & = |\Hess u|^2 - \frac{(\Delta u)^2}{m} = \sum_{i,j=1}^m u_{ij} u_{ij} - \frac{(\Delta u)^2}{m}, \\
		|\nabla u|^2|\mathring{A}|^2 & = |\nabla u|^2|A|^2 - \frac{|\nabla u|^2 H^2}{m-1} = \sum_{i,j=2}^m u_{ij} u_{ij} - \frac{(\Delta u - u_{11})^2}{m-1} \\
		& = \sum_{i,j=2}^m u_{ij} u_{ij} - \frac{(u_{11})^2}{m-1} + \frac{2 (\Delta u) u_{11}}{m-1}  - \frac{(\Delta u)^2}{m-1} \, ,
	\end{align*}
	hence
	\begin{align*}
		|\mathring{\Hess u}|^2 - |\nabla u|^2 |\mathring{A}|^2 & = 2 \sum_{i=2}^m u_{1i} u_{1i} + \frac{m}{m-1} (u_{11})^2 - \frac{2 (\Delta u) u_{11}}{m-1} + \frac{(\Delta u)^2}{m(m-1)} \\
		& = \frac{m-2}{m-1} \sum_{i=2}^m u_{1i} u_{1i} + \frac{m}{m-1} \sum_{i=1}^m u_{1i} u_{1i} - \frac{2(\Delta u) u_{11}}{m-1}  + \frac{(\Delta u)^2}{m(m-1)}
	\end{align*}
	and then the conclusion follows since
	\begin{align*}
		|\nabla^\top|\nabla u||^2 & = \sum_{i=2}^m u_{1i} u_{1i} \, , \\
		|\mathring{\Hess u}(\nu,\,\cdot\,)^\sharp|^2 & = \left|\Hess u(\nu,\,\cdot\,)^\sharp - \frac{\Delta u}{m} \nu\right|^2 \\
		& = |\Hess u(\nu,\,\cdot\,)^\sharp|^2 - \frac{2}{m} (\Delta u) \Hess u(\nu,\nu) + \frac{(\Delta u)^2}{m^2} \\
		& = \sum_{i=1}^m u_{1i} u_{1i} - \frac{2}{m} (\Delta u) u_{11} + \frac{(\Delta u)^2}{m^2} \, .
	\end{align*}
\end{proof}

We also need the following lemma. Before stating it, we recall that under our general assumptions the tensor field
\[
	\frac{\mathring{\Hess u}}{u}
\]
extend continuously to the whole manifold $M$. With a little abuse of notation, we are going to denote such an extension by the same symbol. Define also
\begin{equation}\label{def_lambda}
	\Lambda \doteq \frac{S-\trace Q}{m-1},
\end{equation}
so by \eqref{sub_trace},
\begin{equation} \label{DuL}
	\Delta u + \Lambda u = 0 \qquad \text{in } \, M.
\end{equation}
\begin{lemma}
	If $u$ solves \eqref{sub_static} with $Q\in C^1(M)$, then
	\begin{equation} \label{Hu0_bd}
		\begin{array}{lcl}
		\disp \frac{\mathring{\Hess u}(\nu,\nu)}{u} & = & \disp \frac{(m-1)(m-2)}{2m} \Lambda + \frac{1}{2} \trace Q - \frac{1}{2} S_{\partial M} \\[0.5cm]
		& = & \disp \frac{1}{m} \big(\trace Q - S \big)+\frac{1}{2}(S-S_{\partial M}) \qquad \text{on } \, \partial M \, .
	\end{array}
	\end{equation}
\end{lemma}

\begin{proof}
	From the definition of $Q$ we have
	\[
	\frac{\mathring{\Hess u}}{u} = \mathring{\Ric} - \mathring{Q} \equiv \Ric - \frac{S}{m} g - Q + \frac{\trace Q}{m} g
	\]
	on $\mathring{M}$, and thus on the whole $M$ by continuity of $Q$ and of $\Ric$. Hence,
	\[
	\frac{\mathring{\Hess u}(\nu,\nu)}{u} = \Ric(\nu,\nu) - \frac{S}{m} - Q(\nu,\nu) + \frac{\trace Q}{m} \, .
	\]
	By Gauss equations and from the fact that $\partial M$ is totally geodesic, we have
	\[
	2\Ric(\nu,\nu) = S - S_{\partial M} - |A_{\partial M}|^2 + H_{\partial M}^2 \equiv S - S_{\partial M} \qquad \text{on } \, \partial M \, .
	\]
	On the other hand, under assumptions $Q \in C^1(M)$ we have $\div Q \in C^0(M)$. Thus, by using $S \in C^1(M)$ we can apply \eqref{Qdu} to deduce $Q(\nu,\nu) = 0$ on $\partial M$. Substituting these relations into the above one we obtain
	\[
		\frac{\mathring{\Hess u}(\nu,\nu)}{u} = \frac{m-2}{2m} S + \frac{1}{m} \trace Q - \frac{1}{2} S_{\partial M} \equiv \frac{(m-1)(m-2)}{2m} \Lambda + \frac{1}{2} \trace Q - \frac{1}{2} S_{\partial M} \, .
	\]
\end{proof}

We next assume that 
\[
	S - \trace Q \qquad \text{is constant on } \, M,  
\]
(hence, so is $\Lambda$ in \eqref{def_lambda}). Remarkably, under this assumption the vector field $\mathring{\Hess u}(\nabla u,\,\cdot\,)^\sharp$ happens to be a gradient vector field, namely the gradient of the function
\begin{equation} \label{Fdef}
	F = \frac{1}{2} |\nabla u|^2 + \frac{\Lambda}{2m} u^2
\end{equation}
and, in this setting, the content of \eqref{Sh2} and \eqref{Hu0} can be restated as follows. The next result relates to a series of identities (see, for instance, \cite{foga,cerd2,lind}) arising from a formula obtained by Robinson \cite{robinson} in his investigation of vacuum static black holes. His formula expresses the divergence of $u^{-1}\nabla F$ on a vacuum static $3$-space $M$ in terms of the Cotton tensor of $M$. Later, Shen \cite{sh} rewrote it in a simpler way, relating such divergence to the traceless part of the Ricci tensor. The link between the two expressions essentially depends on the fact that the manifold is vacuum static, and not just substatic. The identity in Lemma \ref{lem_dF} below is closer, in spirit, to Shen's one. In fact, if $Q=0$, the traceless part of $\Ric$ agrees with the traceless part of $u^{-1}\Hess u$, whose norm we related to the geometry of level sets of $u$ in \eqref{Hu0}. When $Q\not\equiv 0$, however, the link between $|\mathring{\Ric}|$ and $|\mathring{\Hess u}|$ is not so neat.

\begin{lemma} \label{lem_dF}
	In the above setting, if \eqref{DuL} holds,
	\[
		\div\left(\frac{\nabla F}{u}\right) = \frac{|\nabla u|^2}{u} |\mathring{A}|^2 + \frac{m-2}{m-1} \frac{|\nabla^\top|\nabla u||^2}{u} + \frac{Q(\nabla u,\nabla u)}{u} + \frac{m}{m-1} \frac{|\nabla F|^2}{u|\nabla u|^2} \qquad \text{on } \, M_0,
	\]
	with $F$ as in \eqref{Fdef} and $M_0 = \{\nabla u \neq 0\}$. In particular, if $Q\geq 0$, then
	\[
		\div\left(\frac{\nabla F}{u}\right) \geq 0 \qquad \text{in } \, M,
	\]
	and the equality holds if and only if $\mathring{\Hess}u=0$.
\end{lemma}

We are ready to prove Theorem \ref{thm_BGH} and Corollary \ref{coro_bgh}. First, because of \eqref{DuL} we have the identity
\begin{equation}\label{eq_with_lambda}
	S_{\partial M} - \frac{m-2}{m}S - \frac{2}{m} \trace Q = S_{\partial M}- \frac{(m-1)(m-2)}{m}\Lambda - \trace Q 
\end{equation}
where $S_{\partial M}$ is the scalar curvature of $\partial M$.

\begin{proof}[Proof of Theorem \ref{thm_BGH}]
	First, observe that by \eqref{eq_with_lambda} inequality \eqref{BGHtype_intro} is equivalent to
	\begin{equation} \label{BGHtype}
				\sum_i \kappa_i^b \int_{\Sigma_i} S_{\Sigma_i} \geq \frac{(m-1)(m-2)}{m} \Lambda \sum_i \kappa_i^b |\Sigma_i| + \sum_i \kappa_i^b \int_{\Sigma_i} \trace Q \, .
	\end{equation}
	By integrating \eqref{DuL} on $M$ against $u$ we get
	\[
		\Lambda \int_M u^2 = \int_M |\nabla u|^2,   
	\]
	thus $\Lambda > 0$, which proves the claimed inequality $S - \trace Q > 0$. The function $F$ defined in \eqref{Fdef} is smooth and positive on the whole of $M$ (note that $F=\frac{1}{2}|\nabla u|^2>0$ on $\partial M$), so the function $F^a$ is also smooth and positive for any exponent $a\in\R$. The vector field
	\[
		X = \frac{(2F)^a\nabla F}{u}
	\]
	is smooth in the interior of $M$ and continous up to $\partial M$, so by the divergence theorem
	\[
		\int_{\partial M} \langle X,\nu\rangle = \int_M \div X
	\]
	where, on $\partial M$, $\nu = - \nabla u/|\nabla u|$ coincides with the outward pointing unit normal. We have
	\[
		\langle X,\nu\rangle = |\nabla u|^{2a} \frac{\mathring{\Hess u}(\nabla u,\nu)}{u} = - |\nabla u|^{2a+1} \frac{\mathring{\Hess u}(\nu,\nu)}{u} \qquad \text{on } \, \partial M,
	\]
	and then by \eqref{Hu0_bd} we get
	\[
		\int_{\partial M} |\nabla u|^{2a+1} S_{\partial M} = \frac{(m-1)(m-2)}{m} \Lambda \int_{\partial M} |\nabla u|^{2a+1} + \int_{\partial M} |\nabla u|^{2a+1} \trace Q + 2\int_M \div X \, .
	\]
	From Lemma \ref{lem_dF} we deduce
	\begin{align*}
		2^{-a} \div X = \div\left(\frac{F^a\nabla F}{u}\right) & = \frac{F^a}{u} \left(|\nabla u|^2|\mathring{A}|^2 + \frac{m-2}{m-1} |\nabla^\top|\nabla u||^2 + Q(\nabla u,\nabla u) \right) \\
		& \phantom{=\;} + \frac{F^{a-1}}{u|\nabla u|^2} \left(\frac{m}{m-1} F + a |\nabla u|^2\right) |\nabla F|^2
	\end{align*}
	on $M_0\cap \mathring{M}$, where $M_0 = \{\nabla u \neq 0\}$. By the definition of $F$, we have
	\[
		\frac{m}{m-1} F + a |\nabla u|^2 = \left(\frac{m}{2(m-1)} + a\right)|\nabla u|^2 + \frac{\Lambda}{2(m-1)} u^2 > 0 \qquad \text{on } \, \mathring{M},
	\]
	as long as $a\geq-\frac{m}{2(m-1)}$, or, equivalently,
	\begin{equation} \label{a_cond}
		2a+1 \geq -\frac{1}{m-1} \, .
	\end{equation}
	Thus, assumption \eqref{a_cond} gives $\div X \geq 0$ on $M_0\cap \mathring{M}$, hence on the whole interior of $M$ by continuity of $\div X$ and density of $M_0$. In fact, note that by $\Delta u + \Lambda u = 0$ the set of critical points $\{|\nabla u|=0\}$ has zero measure, see \cite{uh76} (the conclusion would also follow by observing that $X=0$ on $M\backslash M_0$, and thus $\div X$ vanishes in the interior of $M\setminus M_0$). Consequently, after renaming $b = 2a+1$ we obtain \eqref{BGHtype}. If \eqref{BGHtype} is satisfied with the equality sign, then necessarily $\div X \equiv 0$ on $\mathring{M}$. This implies that $\mathring{A}$, $\nabla^\top|\nabla u|$ and $\nabla F$ vanish everywhere on $M_0\cap \mathring{M}$, so we also have $\mathring{\Hess u} \equiv 0$ on $\mathring{M}$ in view of \eqref{Hu0}. Therefore, 
	\begin{equation}\label{eq_hess_oba}
		\Hess u \equiv -\frac{\Lambda}{m} u g \qquad \text{on } \, M.
	\end{equation}
 	By Reilly's generalization of Obata's theorem to compact manifolds with boundary (see \cite[Lemma 3]{reilly}), $M$ is a round hemisphere of curvature $\Lambda/m$. Hence, \eqref{eq_hess_oba} gives
 	\[
 		u\Ric - \Hess u + (\Delta u) g = (m-1)\frac{\Lambda}{m} u g + \frac{\Lambda}{m}u g - \Lambda u g \equiv 0,
 	\]
	and we conclude $Q \equiv 0$ on $M$.	 
\end{proof}

\begin{proof}[Proof of Corollary \ref{coro_bgh}]
	Having split $\partial M$ into pieces $\{\hat \Sigma_a\}_{a=1}^j$ according to the value of the respective surface gravities $\{\kappa_a\}$, using \eqref{eq_with_lambda} and restricting to $n=3$ inequality \eqref{BGHtype_intro} rewrites as
	\begin{equation}\label{eq_itera}
		\sum_{a=1}^j \kappa_a^b \int_{\hat{\Sigma}_a} \left(S_{\hat{\Sigma}_a} - \frac{2}{3}\Lambda - \trace Q\right) \ge 0.
	\end{equation}
	Dividing by $\kappa_j^b$ and letting $b \to \infty$, only the integral on $\hat{\Sigma}_j$ survives and 
	\[
		\int_{\hat{\Sigma}_j} \left(S_{\hat{\Sigma}_j} - \frac{2}{3}\Lambda - \trace Q\right) \ge 0. 
	\]
	By the Gauss-Bonnet theorem and $Q \ge 0$ we conclude  
	\[
		4\pi \chi(\hat{\Sigma}_j) \ge \frac{2}{3} \Lambda |\hat{\Sigma}_j|. 
	\]
	In particular, $\chi(\hat{\Sigma}_j)>0$ and we deduce  \eqref{ine_enhanced_BGH} after replacing the value of $\Lambda$. Regarding (ii), assume that \eqref{ine_enhanced_BGH_a} holds for $i \le a \le j$, i.e. that
	\[
		4\pi \chi(\hat{\Sigma}_a) = \frac{2}{3} \Lambda |\hat{\Sigma}_a|  \qquad \forall \, i \le a \le j.
	\]
	Then, \eqref{eq_itera} becomes
	\begin{equation}\label{eq_itera_2}
		- \sum_{a =i}^j \kappa_a^b \int_{\hat\Sigma_a} \trace Q + \sum_{a=1}^{i-1} \kappa_a^b \int_{\hat{\Sigma}_a} \left(S_{\hat{\Sigma}_j} - \frac{2}{3}\Lambda - \trace Q\right) \ge 0.
	\end{equation}
	Dividing by $\kappa_j^b$ and letting $b \to \infty$ we get
	\[
		- \int_{\hat \Sigma_j} \trace Q \ge 0, 
	\]
	whence the sub-static condition forces $Q \equiv 0$ on $\hat \Sigma_j$. Iterating the procedure dividing \eqref{eq_itera_2} by $\kappa_{j-1}^b, \ldots, \kappa_i^b$ and letting $b \to \infty$ we deduce that $Q \equiv 0$ on $\hat \Sigma_i \cup \ldots \cup \hat \Sigma_j$, as claimed. Inequality \eqref{eq_itera_2} becomes 
	\[
		\sum_{a=1}^{i-1} \kappa_a^b \int_{\hat{\Sigma}_a} \left(S_{\hat{\Sigma}_j} - \frac{2}{3}\Lambda - \trace Q\right) \ge 0.
	\]
	By repeating the proof of (i) with $i-1$ replacing $j$, we deduce the desired bound in \eqref{ine_enhanced_BGH_am1}. If \eqref{ine_enhanced_BGH_tutti} is satisfied, then by (ii) we get $Q \equiv 0$ on $\partial M$ and equality holds in \eqref{eq_itera}. Thus, by Theorem \ref{thm_BGH}, the manifold $(M^3,g)$ is a round hemisphere and $Q \equiv 0$ on the entire $M$.
\end{proof}

In the context of Theorem \ref{thm_BGH}, it is also possible to establish a constraint on the scalar curvature of the boundary, extending the analysis presented in \cite{ses}. Supposing that all the surface gravities are the same constant \( \kappa_{i} = 1 \) (in particular, this occurs when \( \partial M \) is connected), the function
\[
	F = \frac{1}{2}|\nabla u|^2 + \frac{\Lambda}{2m} u^2
\]
is constant and equal to \( \frac{1}{2} \) on \( \partial M \). Using \eqref{Sh2}, we have
\[
	\diver\left(\frac{\nabla F}{u}\right) = \frac{1}{u}|\mathring{\Hess}u|^2 + \frac{1}{u} Q(\nabla u, \nabla u) \geq 0.
\]
By applying the maximum principle, we conclude that \( F \) attains its maximum on \( \partial M \), unless \( F \) is constant. In the case where \( F \) is constant, it follows directly that \( \mathring{\Hess} u = 0 \), implying that \( M \) is isometric to a hemisphere. Assume that \( F \) achieves its global maximum on \( \partial M \). Since \( u \) reaches its global minimum on \( \partial M \), let \( p \in \partial M \) and, by considering any point \( x \) sufficiently close to \( p \), we have
\[
	0 \geq \langle \nabla F, \nabla u \rangle (x) = u |\nabla u|^2 \left( \frac{\Lambda}{m} + \frac{X(X(u))}{u} \right)(x),
\]
where \( X = \frac{\nabla u}{|\nabla u|} \). Therefore,
\begin{equation} \label{ineq_X}
	\frac{X(X(u))}{u} \leq -\frac{\Lambda}{m}.
\end{equation}
Since \( X(X(u)) = \Hess u(X,X) \), using \eqref{sub_static}, we obtain
\[
	[\Ric(X,X) - \Lambda - Q(X,X)](x) \leq -\frac{\Lambda}{m}.
\]
Taking the limit as \( x \to p \), and applying Gauss' equation along with \eqref{Qdu0}, we get
\[
	S_{\partial M} \geq \Lambda \frac{(m-1)(m-2)}{m} + \trace Q.
\]
Note that, by the rigidity statement in Theorem \ref{thm_BGH} (see also \eqref{BGHtype}), equality holds if and only if \( (M^n, g) \) is isometric to a round hemisphere.

\section{Einstein equations with a map and Potential Sources}\label{sec_map}

In this section, we describe a rigidity result for a static Einstein system whose stress-energy tensor $T$ is given by a static wave map. In this setting, $T$ is given by the metric variation of a natural Lagrangian, as described below. For more details on the Lagrangian formulation in General Relativity, see \cite[Appendix E.1]{wald} or \cite[Section 3.3]{hawking}, for example.

\subsection{Maps between manifolds}
To make computations, we adopt the moving frame formalism introduced by É. Cartan (for more details of this formalism, we refer to \cite[Section 1.7]{amr}).

Let \( \phi: (M^m, g) \to (N^n, h) \) be a smooth map between two Riemannian manifolds. Consider orthonormal frames $\{e_i\}$ and coframes $\{\theta^i\}$ on an open subset $U \subseteq M$, and orthonormal frames $\{E_a\}$ and coframes $\{\omega^a\}$ on an open subset $V \subseteq N$ such that $\varphi^{-1}(V)\subseteq U$. We define
\begin{equation*}
	\varphi^*\omega^a = \varphi^a_i \theta^i,
\end{equation*}
so that the differential \( \di \varphi \), viewed as a $1$-form on $M$ with values in the pullback bundle $\varphi^{-1}TN$, is expressed as
\begin{equation*}
	\di\varphi = \varphi^a_i \theta^i \otimes E_a.
\end{equation*}
The energy density \( e(\varphi) \) of the map $\varphi$ is defined as
\begin{equation}\label{def dens en}
	e(\varphi) = \frac{1}{2} |\di\varphi|^2,
\end{equation}
where $|\di\varphi|^2 = \varphi^a_i \varphi^a_i$. The second fundamental form of the map $\varphi$ is given by $\nabla\di\varphi : TM \otimes TM \to TN$, the covariant derivative of $\di\varphi$ regarded as a section of $T^\ast M \otimes \varphi^{-1}TN$ equipped with the connection $\nabla \otimes \varphi^\ast D$ induced by the Levi-Civita connections $\nabla$ and $D$ of $M$ and $N$, respectively. Explicitly,
\begin{equation*}
	\nabla \di\varphi = \varphi^a_{ij} \theta^j \otimes \theta^i \otimes E_a,
\end{equation*}
where the coefficients \( \varphi^a_{ij} \) are defined by the relation
\begin{equation*}
	\varphi^a_{ij} \theta^j = \di\varphi^a_i - \varphi^a_k \theta^k_i + \varphi^b_i \omega^a_b,
\end{equation*}
and $\{\theta^i_j\}$, $\{\omega^a_b\}$ are the connection forms of $\nabla$ and $D$, respectively.
The tension field \( \tau(\varphi) \) is defined by
\begin{equation}\label{def tension field}
	\tau(\varphi) = \mathrm{tr}  (\nabla \di\varphi)=\varphi_{ii}^aE_{a}.
\end{equation}
In this setting, $|\di \varphi|^2$ satisfies the Bochner formula
\begin{equation} \label{boch}
    \frac{1}{2} \Delta |\di \varphi|^2 = |\nabla \di \varphi|^2 + \varphi_i^a \varphi_{kki}^a + R_{ij} \varphi_i^a \varphi_j^a - R^N_{abcd} \varphi_i^a \varphi_j^b \varphi_i^c \varphi_j^d,
\end{equation}
where $R_{ij}$ and $R^N_{abcd}$ denote the local components of the Ricci tensor of $g$ and the Riemann tensor of $h$ in the given orthonormal frames, respectively. For a proof of this identity, see, for example, \cite[Proposition 1.5]{amr}.

\subsection{The related Einstein's equation}
Let $\Phi:(\hat{M}^{m+1},\hat{g}) \longrightarrow (N^n,h)$ and $V:(N^n,h)\longrightarrow\mathbb{R}$ be two smooth maps. Consider the matter Lagrangian defined by:
\[
	\mathcal{L}(\hat{g},\Phi) = \int_{\hat{M}} \left[ |\di\Phi|_{\hat{g}}^2 + (m-1)V(\Phi) \right] \di x_{\hat{g}}.
\]
The stress-energy tensor associated to this system is obtained by varying $\mathcal{L}$ with respect to $\hat{g}$, which gives
\[
	T = \Phi^{*}h - \frac{1}{2} \left( |\di \Phi|_{\hat{g}}^2 + (m-1)V(\Phi) \right)\hat{g}.
\]
Therefore, Einstein's equation \eqref{Einst_eq} can be written as
\begin{equation}\label{eq_maps}
	\Ric_{\hat{g}}+\Lambda\hat{g} = \Phi^{*}h + V(\Phi)\hat{g}.
\end{equation}
We observe that the cosmological constant can be incorporated in the function $V$ by adding a constant in $V$. Thus, without loss of generality, we assume $\Lambda = 0$ in the subsequent analysis. The equation of motion for this system, obtained as the Euler Lagrangian equation with respect to $\Phi$ is given by
\[
	\hat{\tau}(\Phi) = \frac{m-1}{2} D V(\Phi),
\]
where $\hat{\tau}(\Phi)$ denotes the tension field. For detailed computations, see \cite[Section 5]{ans21}. For the special case where $N=\mathbb{R}$, we refer to \cite{reiris}.

We now focus on the static case, where $(\hat{M}, \hat{g})$ is given by \( \hat{M} = \mathbb{R} \times M \) and \( \hat{g} = -u^2 \, \mathrm{d}t \otimes \mathrm{d}t + g \), with $(M^m, g)$ Riemannian and $0<u \in C^\infty(M)$. We assume that $\Phi$ is static as well, that is, it factorizes through the projection \(\pi:\hat{M}\to M\) as the composition  \(\Phi = \phi \circ \pi \), for some map \( \phi: M \to (N^n, h) \).

In this framework, system \eqref{eq_maps} becomes
\begin{equation}\label{mapfieldequation}
	\left\{
		\begin{array}{r@{\;}c@{\;}l}
		\Ric  - \frac{\Hess  u}{u} & = & \varphi^{*}h+V(\phi)g \\[0.2cm]
		-\Delta  u & = & V(\phi) u.
		\end{array}
	\right.
\end{equation}
Additionally, given that \(u \hat{\tau}(\Phi) = \di\phi(\nabla u) + u\tau(\phi)\), motion equation becomes
\begin{equation}\label{motion_eq_map}
	u\tau(\varphi)+\di\varphi(\nabla u)=(m-1)\frac{D V(\phi)}{2}u.
\end{equation}

By introducing the known change of variable \( u = e^{-f} \), the coupled system \eqref{mapfieldequation}-\eqref{motion_eq_map} transforms into the following coupled system:
\begin{equation}\label{eq_field_f}
	\left\{
		\begin{array}{r@{\;}c@{\;}l}
			\Ric_f^{m+1} & = & \varphi^{*}h + V(\varphi)g \\[0.2cm]
			\Delta_f f & = & V(\varphi) \\[0.2cm]
			\tau(\varphi) - \di\varphi(\nabla f) & = & \frac{m-1}{2} D V(\varphi)
		\end{array}
	\right.
\end{equation}
where we set
\[
	\Delta_f \psi \doteq \Delta \psi - g(\nabla f, \nabla \psi).
\]
Note that the components of the weighted Ricci curvature \(\Ric_f^{m+1}\) are given by
\[
	(\Ric_f^{m+1})_{ij} = R_{ij} + f_{ij} - f_i f_j = R_{ij}^{f} - f_i f_j,
\]
where \( R_{ij}^{f} \) are the components of the Bakry-Émery Ricci tensor. 

The third equation in \eqref{eq_field_f} suggests to define a weighted operator related to the tension field:
$$
	\tau_{f}(\varphi)=\tau(\varphi)-\di\varphi(\nabla f). 
$$
In coordinates,
	$$(\varphi_{kk}^{a})^{f}=\varphi_{kk}^{a}-\varphi_{j}^{a}f_{j}.$$
We have the following Bochner's formula:

\begin{lemma}\label{bochner}
	\begin{equation}
		\frac{1}{2}\Delta_{f}|\di\varphi|^{2}=|\nabla \di\varphi|^{2}+\langle \nabla \tau_{f}(\varphi),\di\varphi\rangle_{N}+Q(\di\varphi),    
	\end{equation}
	where 
	$$
		Q(\di\varphi)=R_{ij}^{f} \varphi^a_i \varphi^a_j - R^N_{abcd} \varphi^a_i \varphi^b_j \varphi^c_i \varphi^d_j.
	$$
\end{lemma}

\begin{proof}
	The identity follows by coupling the Bochner formula \eqref{boch} with the following identity
	\begin{eqnarray*}
		\varphi_{i}^{a}[(\varphi_{kk}^{a})^{f}]_{i}+R_{ij}^{f} \varphi^a_i \varphi^a_j &=&\varphi_{i}^{a}\varphi_{kki}^{a}-\varphi_{i}^{a}(\varphi_{j}^{a}f_{j})_{i}+R_{ij} \varphi^a_i \varphi^a_j+f_{ij}\varphi^a_i \varphi^a_j\\
		&=&\varphi_{i}^{a}\varphi_{kki}^{a}+R_{ij} \varphi^a_i \varphi^a_j-\varphi_{i}^{a}\varphi_{ij}^{a}f_{j}\\
		&=&\varphi_{i}^{a}\varphi_{kki}^{a}+R_{ij} \varphi^a_i \varphi^a_j-\frac{1}{2} f_j(|\di \phi|^2)_j.
	\end{eqnarray*}
\end{proof}

We now consider the more general setting of Theorem \ref{vanishing_map}, where \eqref{eq_field_f} is replaced by \eqref{map_source}. The first equation in \eqref{map_source} implies
$$
	{(R_{ij}^{f})^{m+1}}\varphi_{i}^{a}\varphi_{j}^{a} \ge \varphi^b_i\varphi^b_j\varphi^a_i\varphi^a_j + V(\varphi)|\di\varphi|^2,
$$
while the third one in \eqref{map_source} gives
$$
	(\varphi_{kk}^a)^{f} = \frac{m-1}{2} V^a(\varphi) \qquad \mbox{and} \qquad[(\varphi_{kk}^{a})^{f}]_{i}=\frac{m-1}{2} V^a_b(\varphi)\varphi^b_i.
$$
Inserting into the Bochner's formula in Lemma \ref{bochner} we get
\begin{eqnarray*}
	\frac{1}{2}\Delta_{f}|\di \varphi|^{2}& \ge &|\nabla \di \varphi|^{2}+\frac{m-1}{2}V_{ab}(\varphi) \varphi^a_i \varphi^b_i +\varphi^a_i\varphi^a_j\varphi^b_i\varphi^b_j+V(\varphi)|\di \varphi|^{2}\\
	& & +f_{i}\varphi_{i}^{a}f_{j}\varphi_{i}^{a}- R^N_{abcd} \varphi^a_i \varphi^b_j \varphi^c_i \varphi^d_j\\
	&\geq & \left[\frac{m-1}{2}\Hess V+Vh\right]_{ab}(\varphi) \varphi^a_i\varphi^b_i+\varphi^a_i\varphi^a_j\varphi^b_i\varphi^b_j-R^N_{abcd} \varphi^a_i \varphi^b_j \varphi^c_i \varphi^d_j.
\end{eqnarray*}
Hence,
\begin{equation}\label{ineq_liouv}
	\frac{1}{2}\Delta_{f}|\di \varphi|^{2}\geq\left[\frac{m-1}{2}\Hess V+Vh\right]_{ab}(\varphi) \varphi^a_i\varphi^b_i+Q_{0}(\di \varphi),
\end{equation}
where we set
$$
	Q_{0}(\di \varphi)=\varphi^a_i\varphi^a_j\varphi^b_i\varphi^b_j-R^N_{abcd} \varphi^a_i \varphi^b_j \varphi^c_i \varphi^d_j.
$$

We shall examine $Q_{0}$. 
\begin{lemma} \label{lemma_estimate}
    If 
    $$
    	\sup_{N}\sec_{N}\leq\kappa\leq\frac{1}{m-1}
    $$
	for some $\kappa \ge 0$, then 
    $$
    	Q_{0}(\di\varphi) \geq \frac{1-(m-1)\kappa}{m} |\di\varphi|^4.
    $$
\end{lemma}

\begin{proof}
    We decompose
    \[
        Q_0(\di\varphi) = Q_1(\di\varphi) + Q_2(\di\varphi),
    \]
    where
    \begin{align*}
        Q_1(\di\varphi) & = (m-1)\kappa \varphi^a_i \varphi^a_j \varphi^b_i \varphi^b_j - R^N_{abcd} \varphi^a_i \varphi^b_j \varphi^c_i \varphi^d_j, \\
        Q_2(\di\varphi) & = (1 - (m-1)\kappa) \varphi^a_i \varphi^a_j \varphi^b_i \varphi^b_j = (1 - (m-1)\kappa) \|\varphi^\ast h\|^2.
    \end{align*}
    Since \( \sec_N \leq \kappa \), we have \( Q_1(\di\varphi) \geq 0 \) by \cite[Lemma 2.2]{cmrigoli}. On the other hand, since \( 1 - (m-1)\kappa \ge 0 \) by assumption, applying the Newton inequality we obtain
    \[
        \|\varphi^\ast h\|^2 \geq \frac{(\mathrm{tr} \, \varphi^\ast h)^2}{m} = \frac{1}{m} |\di\varphi|^4,
    \]
    which implies
    \[
        Q_0(\di\varphi) \geq Q_2(\di\varphi) \geq \frac{1 - (m-1)\kappa}{m} |\di\varphi|^4.
    \]
\end{proof}

From Lemma \ref{lemma_estimate}, if $\displaystyle\sup_{N}\sec_{N}\leq\kappa < \frac{1}{m-1}$ for some constant $\kappa \ge 0$, the map $\varphi$ satisfies the following inequality:
\begin{equation}\label{eq_dphi}
	\frac{1}{2}\Delta_{f}|\di\varphi|^{2} \geq \left[\frac{m-1}{2}\Hess V + Vh \right]_{ab}(\varphi) \varphi^a_i \varphi^b_i + \frac{1 - (m-1)\kappa}{m}|\di\varphi|^{4}.
\end{equation}
In view of Keller-Osserman's theory (see \cite{bmpr} for a detailed account) a Liouville-type result for the map $\varphi$ may follow under suitable hypotheses on the potential $V$ and a control of the $f$-volume of balls. Given a complete Riemannian manifold \( M \) we fix an origin \( \mathcal{O}\subset M \) as either a single point or a relatively compact, open subset with a smooth boundary \( \partial\mathcal{O} \). We define the function \( r(x) = \operatorname{dist}(x, \mathcal{O}) \), and for \( R > 0 \), the ball
\[
	B_{R}(\mathcal{O}) = \left\{x \in M; r(x) \in (0, R)\right\}.
\]
The next result could be viewed as a version of \cite[Lemma 4.6]{rigoli2} to (weighted) manifolds with boundary. In general, \cite[Lemma 4.6]{rigoli2} is not expected to hold for manifolds with boundary. The main point here is that, in our setting, the lapse function $u$ (hence, the density $e^{-f}$ of the weighted measure) vanishes on $\partial M$, allowing us to get rid of the boundary terms. 

In what follows, we set
\[
	\di \mu_f = e^{-f}\di x, \qquad \vol_f(A) = \int_A \di \mu_f \qquad \forall A \subset M.
\]

\begin{proposition}\label{vol_control1}
	Let $(M^{m},g,u)$ be a complete manifold with $\partial M=u^{-1}(0)$ and $u>0$ in $\mathring{M}$. Let $v\in {\rm Lip}_{\rm{loc}}(M)$ and $\gamma>0$ such that
	$$
		\Omega_{\gamma} = \left\{x\in M:\:\ v(x)>\gamma\right\} \neq \emptyset.
	$$
	For $f=-\ln{u}$, suppose that $v$ satisfies
	\begin{equation}\label{lap_estimate}
		\Delta_{f}v\geq bv^{\sigma}\qquad\mbox{on}\quad\Omega_{\gamma},    
	\end{equation}
	for positive constants $b>0$ and $\sigma>1$. If 
	\begin{equation}\label{hip_volume}
		\displaystyle\liminf_{r\to\infty}\frac{\ln{\vol_f ( B_{r}(\mathcal{O}))}}{r^2}<\infty, 
	\end{equation}
	then
	$$
		\sup_M v < \infty.
	$$
\end{proposition}

\begin{proof}
	The proof closely follows that of \cite[Lemma 4.5]{rigoli2}. We decide to provide full details to underline the role of the boundary terms. The core step is the following growth inequality\\[0.2cm]
	\noindent\textbf{Claim 1}: There exists a constant $C>0$ such that, for every $r>0$ and $\alpha>1$,
	\begin{equation}\label{bound_eq}
		\vol_{f}(\Omega_{\gamma}\cap B_{r}(\mathcal{O})) \leq\left[\frac{1}{\gamma}\frac{C}{r^{2}}\frac{1}{b}\frac{(\alpha+\sigma-1)^{2}}{\alpha-1}\right]^{\alpha+\sigma-1}\vol_{f}(\Omega_{\gamma}\cap B_{2r}(\mathcal{O})).
	\end{equation}
	\begin{proof}[Proof of Claim 1:]
		Fix a constant $\zeta>1$ such that
		$$
			2+\frac{2}{\sigma-1}\left(\frac{1}{\zeta}-1\right)>0.
		$$
		Choose a cut-off function $\psi:M\longrightarrow [0,1]$ such that
		\begin{enumerate}
			\item $\psi\equiv 1$ on $B_{r}$;
			\item $\psi\equiv 0$ on $M\setminus B_{2r}$; 
			\item $|\nabla\psi|\leq\frac{C_{0}}{r}\psi^{\frac{1}{\zeta}}$, for a constant $C_{0}=C_{0}(\zeta)>0$.
		\end{enumerate}
		Fix constants $\alpha>1$, $\varepsilon>0$ and a $C^{1}$ non-decreasing function $\lambda:\mathbb{R}\longrightarrow [0,\infty)$ such that $\lambda(t)=0$ for $t\leq\gamma$, and test the equation \eqref{lap_estimate} against the function
		$$
			\psi^{2(\alpha+\sigma-1)}\lambda(v)v^{\alpha-1}\eta_{\varepsilon}(u),
		$$
		where
		$$
			\eta_{\varepsilon}(u)=\left\{
				\begin{array}{l@{\;}c@{\;}l}
					0 & {\rm if} & u\leq\varepsilon \\[0.2cm]
					\frac{u-\varepsilon}{\varepsilon}, & {\rm if} & \varepsilon\leq u\leq 2\varepsilon \\[0.2cm]
					1, & {\rm if} & u\geq 2\varepsilon.
				\end{array}
			\right.
		$$
		Observe that 
		\begin{eqnarray*}
			& &\int\psi^{2(\alpha+\sigma-1)}\lambda(v)v^{\alpha+1}\langle\nabla u,\nabla\eta_{\varepsilon}(u)\rangle \di x_{f}\\
			&=&\int_{\left\{\varepsilon<u<2\varepsilon\right\}}\psi^{2(\alpha+\sigma-1)}\lambda(v)v^{\alpha+1}\frac{|\nabla u|^{2}}{\varepsilon}\di x_{f} \to 0\quad{\rm as} \quad\varepsilon\to 0
		\end{eqnarray*}
		which accounts for the fact that the measure density $e^{-f}=u$ vanishes on $\partial M$. Therefore, by letting $\varepsilon\to 0$ and using $\lambda'\geq 0$ we obtain
		\begin{eqnarray*}
			\displaystyle\int \psi^{2(\alpha+\sigma-1)}\lambda(v)bv^{\alpha+\sigma-1}\di x_{f}&\leq& -(\alpha-1)\displaystyle\int\psi^{2(\alpha+\sigma-1)}\lambda(v)v^{\alpha-2}|\nabla v|^{2}\di x_{f}\\
			& &+2(\alpha+\sigma-1)\displaystyle\int\psi^{2(\alpha+\sigma-1)-1}\lambda(v)v^{\alpha-1}|\nabla\psi||\nabla v|\di x_{f}. 
		\end{eqnarray*}
		By Young's inequality, we can estimate
		\begin{eqnarray*}
			& &2(\alpha+\sigma-1)\displaystyle\int\psi^{2(\alpha+\sigma-1)-1}\lambda(v)v^{\alpha-1}|\nabla\psi||\nabla v|\di x_{f} \\
			&=&\int\left[\psi^{\alpha+\sigma-1}\lambda(v)^{\frac{1}{2}}v^{\frac{\alpha-2}{2}}|\nabla u|\right]\left[2\psi^{\alpha}\lambda(v)^{\frac{1}{2}}v^{\frac{\alpha}{2}}|\nabla\psi|(\alpha+1)\right]\di x_{f}\\
			&\leq&(\alpha-1)\displaystyle\int\psi^{2(2\alpha+\sigma-1)}\lambda(v)v^{\alpha-2}|\nabla v|^{2}\di x_{f}+\frac{(\alpha+\sigma-1)^{2}}{\alpha-1}\displaystyle\int\psi^{2(\alpha+\sigma-1)-2}\lambda(v)v^{\alpha}|\nabla\psi|^{2}\di x_{f}.
		\end{eqnarray*}
		Thus,
		\begin{eqnarray}
			\displaystyle\int \psi^{2(\alpha+\sigma-1)}\lambda(v)bv^{\alpha+\sigma-1}\di x_{f}&\leq&\displaystyle\int\psi^{2(\alpha+\sigma-1)}\lambda(v)v^{\alpha}\di x_{f}\nonumber\\
			&&+\frac{(\alpha+\sigma-1)^2}{\alpha-1}\displaystyle\int\psi^{2(\alpha+\sigma-1)-2(1-\frac{1}{\zeta})}\lambda(v)v^{\alpha}(\psi^{-\frac{1}{\zeta}}|\nabla\psi|)^{2}\di x_{f}.\label{estimate1}
		\end{eqnarray}
		We use Hölder inequality (with coeficients $p,q$ to be chosen) in order to estimate the last integral of the (RHS) that we named $(I)$. In fact,
		\begin{eqnarray*}
			(I)&=&\displaystyle\int\left\{\left[\psi^{2(\alpha+\sigma-1)}\lambda(v)b\right]^{\frac{1}{p}}v^{\alpha}\right\}\left\{\left[\psi^{2(\alpha+\sigma-1)}\lambda(v)b\right]^{\frac{1}{q}}b^{-1}\psi^{-2(1-\frac{1}{\zeta})}\right\}\left\{\psi^{-\frac{1}{\zeta}}|\nabla\psi|\right\}^{2}\di x_{f}\\
			&\leq&\frac{C_{0}^{2}}{r^{2}}\left[\displaystyle\int\psi^{2(\alpha+\sigma-1)}\lambda(v)bv^{\alpha p}\di x_{f}\right]^{\frac{1}{p}}\left[\displaystyle\int \psi^{2(\alpha+\sigma-1)+2q(\frac{1}{\zeta}-1)}\lambda(v)b^{1-q}\di x_{f}\right]^{\frac{1}{q}}.
		\end{eqnarray*}
		Choosing
		$$
			p=\frac{\alpha+\sigma-1}{\alpha}\quad \mbox{and}\quad q=\frac{\alpha+\sigma-1}{\sigma-1},
		$$
		we have
		$$
			2(\alpha+\sigma-1)+2q\left(\frac{1}{\zeta}-1\right)=(\alpha+\sigma-1)\left[2+\frac{2}{\sigma-1}\left(\frac{1}{\zeta}-1\right)\right]>0,
		$$
		by our choice of $\zeta$. Since $\psi\leq 1$ and $\psi\equiv 0$ off $B_{2r}$, we conclude
		\begin{eqnarray*}
			(I)&\leq&\frac{C_{0}^{2}}{r^{2}}\left[\displaystyle\int\psi^{2(\alpha+\sigma-1)}\lambda(v)bv^{\alpha+\sigma-1}\di x_{f}\right]^{\frac{\alpha}{\alpha+\sigma-1}}\left[\displaystyle\int_{B_{2r}}\lambda(v)b^{-\frac{\alpha}{\sigma-1}}\di x_{f}\right]^{\frac{\sigma-1}{\alpha+\sigma-1}}\\
			&=&\frac{C_{0}^2}{r^{2}b}\left[\displaystyle\int\psi^{2(\alpha+1)}\lambda(v)bv^{\alpha+1}\di x_{f}\right]^{\frac{\alpha}{\alpha+\sigma-1}}\left[\displaystyle\int_{B_{2r}}\lambda(v)b\di x_{f}\right]^{\frac{\sigma-1}{\alpha+\sigma-1}}.    
		\end{eqnarray*}
		Then, replacing in \eqref{estimate1}, we have
		\begin{eqnarray*}
			\left[\displaystyle\int\psi^{2(\alpha+\sigma-1)}\lambda(v)bv^{\alpha+\sigma-1}\di x_{f}\right]^{\frac{\sigma-1}{\alpha+\sigma-1}}\leq \frac{C_{0}^2}{r^{2}b}\frac{(\alpha+\sigma-1)^2}{\alpha-1} \left[\displaystyle\int_{B_{2r}}\lambda(v)b\di x_{f}\right]^{\frac{\sigma-1}{\alpha+\sigma-1}}   
		\end{eqnarray*}
		To conclude we just estimate the (LHS) from below since $\psi\equiv 1$ on $B_{r}$ and $\lambda(v)=0$ if $v\leq \gamma$:
		$$
			\displaystyle\int\psi^{2(\alpha+\sigma-1)}\lambda(v)bv^{\alpha+\sigma-1}\di x_{f}\geq \displaystyle\int_{B_{r}\cap\Omega_{\gamma}}\psi^{2(\alpha+\sigma-1)}\lambda(v)bv^{\alpha+\sigma-1}\di x_{f}\geq \gamma^{\alpha+\sigma-1}\int_{B_{r}\cap \Omega_{\gamma}}b\lambda(v)\di x_{f}.
		$$
		Then
		$$
			\int_{B_{r}(\mathcal{O})}\lambda(v)\di x_{f}\leq \left[\frac{1}{\gamma}\frac{C}{r^{2}}\frac{1}{b}\frac{(\alpha+\sigma-1)^{2}}{\alpha-1}\right]^{\alpha+\sigma-1}\int_{B_{2r}(\mathcal{O})}\lambda(v)\di x_{f}.
		$$
		The thesis follows by choosing an increasing sequence of functions $\lambda_{j}$ pointwise converging to the characteristic function of $\left\{t>\gamma\right\}$.
	\end{proof}
	We are ready to prove that 
	$$
		\sup v<\infty.
	$$
	Assume by contradiction that this is not the case, so $\Omega_{\gamma}\neq\emptyset$ for each $\gamma>0$. Fix $\gamma$, define
	$$
		G(r)= \vol_{f}(\Omega_{\gamma}\cap B_{r}(\mathcal{O})),
	$$
	and choose $R>0$ and $\alpha$ so that 
	$$
		G(r)>0, \quad \alpha+\sigma-1=\frac{b\gamma r^{2}}{8C}\quad \forall r>R.
	$$
	Then by Claim 1,
	$$
		G(r)\leq \left[\frac{1}{\gamma}\frac{C}{r^{2}}\frac{1}{b}\frac{(\alpha+\sigma-1)^{2}}{\alpha-1}\right]^{\alpha+\sigma-1}G(2r)\leq \left[\frac{4C}{b\gamma r^{2}}(\alpha+\sigma-1)\right]^{\alpha+\sigma-1}G(2r)=\left(\frac{1}{2}\right)^{\frac{b\gamma}{8 C}r^2}G(2r),
	$$
	for $r>R$. By iterating (see \cite[Lemma 4.7]{rigoli2}), we can conclude
	$$
		\displaystyle\liminf_{r\to\infty}\frac{1}{r^2}\ln{\vol_{f}(B_{r}(\mathcal{O}))}\geq S \gamma \ln{2},
	$$
	where $S$ is a positive constant independent of $\gamma$. This contradicts the assumption if $\gamma$ is large enough.
\end{proof}

The second proposition we need is the following weak maximum principle at infinity:

\begin{proposition}\label{prop_Fsup} 
	Let $(M^{m},g,u)$ be a complete manifold with $\partial M=u^{-1}(0)$ and $u>0$ in $\mathring{M}$. Let $v\in {\rm Lip}_{\rm{loc}}(M)$  solve
	\begin{equation}\label{lap_estimate1}
		\Delta_{f}v\geq F(v)\qquad\mbox{on}\quad\Mo   
	\end{equation}
	for some $F\in C(\mathbb{R})$, and where $f=-\ln{u}$. If 
	\begin{equation}\label{hip_volume}
		\displaystyle\liminf_{r\to\infty}\frac{\ln{\vol_{f}(B_{r}(\mathcal{O}))}}{r^2}<\infty, 
	\end{equation}
	for some origin $\mathcal{O}\subset M$ and 
	$$
		\sup_{M}v<\infty,
	$$
	then
	$$
		F(\sup_{M} v)\leq 0
	$$
\end{proposition}

Proposition \ref{prop_Fsup} is a version, for weighted operators and manifolds with boundary, of \cite[Theorem 4.1]{rigoli2}. The proof follows verbatim the one therein once we observe the following:
\begin{itemize}
	\item as in Proposition \ref{vol_control1}, the introduction of the weight $e^{-f}$ in measures and operators does not make any difference, and since $e^{-f}=u=0$ on $\partial M$, boundary terms do not appear;
	\item the regularity $v\in C^{1}(M)$ required in the theorem can be weakened to $v\in {\rm Lip}_{\rm{loc}}(M)$ by using the weak formulation of \eqref{lap_estimate1}.
\end{itemize}
Putting together Propositions \ref{vol_control1} and \ref{prop_Fsup} we obtain the next result.

\begin{theorem}\label{liouville_type}
	Let $(M^{m},g,u)$ be a complete manifold with $\partial M=u^{-1}(0)$ and $u>0$ in $\mathring{M}$. Let $v\in {\rm Lip}_{\rm{loc}}(M)$ solve
	\begin{equation}\label{lap_estimate}
		\Delta_{f}v\geq bv^{\sigma}-av\qquad\mbox{on}\quad\Mo   
	\end{equation}
	for constants $b>0$, $a\geq 0$ and $\sigma>1$, where $f=-\ln{u}$. If 
	\begin{equation}\label{hip_volume}
		\displaystyle\liminf_{r\to\infty}\frac{\ln{\vol_{f}(B_{r}(\mathcal{O}))}}{r^2}<\infty, 
	\end{equation}
	for some origin $\mathcal{O}\subset M$, then
	$$
		\displaystyle\sup_{M} v\leq \left(\frac{a}{b}\right)^{\frac{1}{\sigma-1}}.
	$$
\end{theorem}

\begin{proof}
	We first prove that $\sup_M u < \infty$. By contradiction, assume the contrary and pick  
	$$
		\gamma > \left(\frac{a}{b}\right)^{\frac{1}{\sigma-1}}.
	$$
	Then, there exists a constant $C_{\gamma}>0$ such that
	$$
		\Delta_{f}v\geq C_{\gamma}v^{\sigma}\qquad\mbox{on}\quad \Omega_{\gamma},   
	$$
	whence $\sup_M v<\infty$ by Proposition \ref{vol_control1},  contradiction. Using Proposition \ref{prop_Fsup} with $F(t)=bt^{\sigma}-at$ we conclude the thesis.
\end{proof}

\begin{remark}
	We stress that in the above result $\partial M$ may be non-compact, and no assumption on $v_{|_{\partial M}}$ is made besides its local Lipschitz coninuity.
\end{remark}
We eventually prove Theorem \ref{vanishing_map}. Since $v=|\di \varphi|^{2}$ satisfies \eqref{eq_dphi}, which in our assumptions imply
$$\Delta_{f}v\geq -av+\frac{2}{m}(1-(m-1)\kappa)v^{2},$$
the result is a direct application of Theorem \ref{liouville_type} once we prove \eqref{hip_volume}. This folows by standard comparison theory. Consider a smooth relatively compact open set $\mathcal{O}$ containing $\partial M$. By \eqref{map_source} and since $V(\varphi)$ is bounded from below, there exists a constant $\bar{\kappa}>0$ such that 
$$\Ric_{f}^{m+1}\geq -\bar{\kappa}^{2}m g.$$
Fix a point $x\in M\setminus\mathcal{O}$, which is not in the cut locus ${\rm cut}(\mathcal{O})$, and  a unit speed minimizing geodesic $\gamma:[0,r(x)]\to M$ from $\partial\mathcal{O}$ to $x$. We have the following weighted Bochner formula for the $f$-laplacian of $r$ (see \cite[Lemma 2.2]{mrs})
\[
\frac{1}{m} (\Delta_{f} r)^2 + \langle \nabla \Delta_{f} r, \nabla r \rangle + \Ric_{f}^{m+1} (\nabla r, \nabla r) \leq 0,
\]
which implies
$$z'+\frac{z^{2}}{m}-m\bar{\kappa}^{2}\leq 0,$$
where $z(t)=\Delta_{f}r(\gamma(t))$. Furthermore 
$$z(0)\leq\bar{\lambda}=\max_{\partial\mathcal{O}}H_{f},$$
where $H_{f}$ is the weighted mean curvature of $\partial\mathcal{O}$ in the direction pointing towards $\mathcal{O}$. The classical Riccati comparison (see \cite[Theorem 3.8]{bmpr} for a version of this inequality using the distance from a compact set, as described above) gives 
$$\Delta_{f}r\leq m\frac{h'(r)}{h(r)}\qquad {\rm at }\, x,$$
where $h$ is a solution of
\begin{equation}
\left\{ \begin{array}{ll}
h''-\bar{\kappa}^{2}h= 0 & \text{on}\quad \R^+, \\[0.2cm]
h(0)=1, & h'(0)=\bar{\lambda}. \end{array}
\right.
\end{equation}
The inequality also hold in the weak sense, see (see \cite[Lemma 2.5]{rigoli3}). 
By integration and since $h=\cosh (\bar{\kappa}t)+ \frac{\bar{\lambda}}{\bar{\kappa}}\sinh(\bar{\kappa}t)$ we obtain the volume inequality (see \cite[Theorem 2.14]{rigoli3})
\[
\vol_{f}(B_{r}(\mathcal{O}) \setminus B_{r_{0}}(\mathcal{O})) \leq C \left[\int_{r_{0}}^{r} h^{m-1}(t) \, dt\right]\leq C_{1}e^{C_{2}r}
\]
for sonstants $C_1,C_2>0$. The desired inequality \eqref{hip_volume} follows at once.

\vspace{0.3cm}

\noindent \textbf{Acknowledgements.} A.F. would like to thank the hospitality of the Mathematics Department of Università degli Studi di Torino, where part of this work was carried out. G.C. acknowledges that part of this work was carried out while he was employed at the Mathematics Department of Università degli Studi di Napoli Federico II\@. A.F. was partially supported by Conselho Nacional de Desenvolvimento Científico e Tecnológico (CNPq) of the Ministry of Science, Technology and Innovation of Brazil, Grants 316080/2021-7 and 200261/2022-3. This work also was funded by Paraíba State Research Foundation - Programa Primeiros Projetos (Grant 2021/3175) and by Program Research in Pairs CIMPA-ICTP 2024. L.M. and M.R. were supported by the PRIN2022 project no. 20225J97H5 ``Differential-geometric aspects of manifolds via Global Analysis''.

The authors are thankful to the anonymous referee for their thorough review, that led to substantial improvement of the manuscript.

\bigskip

\bibliographystyle{plain}

\end{document}